\pdfoutput=1
\newif\ifpersonal
\documentclass[10pt,a4paper,reqno]{amsart} 
\usepackage{amsmath,amsthm,amssymb,mathrsfs,mathtools,bm,eucal} 
\usepackage{microtype,lmodern} 
\usepackage[utf8]{inputenc} 
\usepackage[T1]{fontenc} 
\usepackage{enumerate,xspace,tikz-cd,csquotes} 
\tikzcdset{cells={font=\everymath\expandafter{\the\everymath\displaystyle}}} 

\usepackage[centering,vscale=0.7,hscale=0.75]{geometry}
\usepackage[hidelinks]{hyperref}
\usepackage[capitalize]{cleveref}
\usepackage{tensor} 

\theoremstyle{plain}
\newtheorem{thm-intro}{Theorem}
\newtheorem{thm}{Theorem}[section]
\newtheorem*{thm*}{Theorem}
\newtheorem{lem}[thm]{Lemma}
\newtheorem{prop}[thm]{Proposition}

\newtheorem{cor}[thm]{Corollary}

\theoremstyle{definition}
\newtheorem{defin}[thm]{Definition}
\newtheorem{notation}[thm]{Notation}
\newtheorem{eg}[thm]{Example}
\newtheorem{rem}[thm]{Remark}
\numberwithin{equation}{section}
\newtheorem{construction}[thm]{Construction}

\ifpersonal
\newcommand{\personal}[1]{\textcolor[rgb]{0,0,1}{(Personal: #1)}}
\newcommand{\todo}[1]{\textcolor{red}{(Todo: #1)}}
\else
\newcommand*{\personal}[1]{\ignorespaces}
\newcommand*{\todo}[1]{\ignorespaces}
\fi

\newcommand{\rL}{\mathrm L}
\newcommand{\rR}{\mathrm R}

\newcommand{\fF}{\mathfrak F}
\newcommand{\fG}{\mathfrak G}

\newcommand{\fU}{\mathfrak U}

\newcommand{\fX}{\mathfrak X}
\newcommand{\fY}{\mathfrak Y}
\newcommand{\fZ}{\mathfrak Z}

\newcommand{\ff}{\mathfrak f}

\newcommand{\cC}{\mathcal C}
\newcommand{\cD}{\mathcal D}

\newcommand{\cF}{\mathcal F}
\newcommand{\cH}{\mathcal H}
\newcommand{\cG}{\mathcal G}

\newcommand{\cK}{\mathcal K}

\newcommand{\cO}{\mathcal O}

\newcommand{\cQ}{\mathcal Q}

\newcommand{\cS}{\mathcal S}
\newcommand{\cT}{\mathcal T}

\newcommand{\cX}{\mathcal X}
\newcommand{\cY}{\mathcal Y}

\DeclareFontFamily{U}{BOONDOX-calo}{\skewchar\font=45 }
\DeclareFontShape{U}{BOONDOX-calo}{m}{n}{<-> s*[1.05] BOONDOX-r-calo}{}
\DeclareFontShape{U}{BOONDOX-calo}{b}{n}{<-> s*[1.05] BOONDOX-b-calo}{}
\DeclareMathAlphabet{\mathcalboondox}{U}{BOONDOX-calo}{m}{n}

\newcommand{\bA}{\mathbf A}

\makeatletter
\let\save@mathaccent\mathaccent
\newcommand*\if@single[3]{%
	\setbox0\hbox{${\mathaccent"0362{#1}}^H$}%
	\setbox2\hbox{${\mathaccent"0362{\kern0pt#1}}^H$}%
	\ifdim\ht0=\ht2 #3\else #2\fi
}
\newcommand*\rel@kern[1]{\kern#1\dimexpr\macc@kerna}
\newcommand*\widebar[1]{\@ifnextchar^{{\wide@bar{#1}{0}}}{\wide@bar{#1}{1}}}
\newcommand*\wide@bar[2]{\if@single{#1}{\wide@bar@{#1}{#2}{1}}{\wide@bar@{#1}{#2}{2}}}
\newcommand*\wide@bar@[3]{%
	\begingroup
	\def\mathaccent##1##2{%
		\let\mathaccent\save@mathaccent
		\if#32 \let\macc@nucleus\first@char \fi
		\setbox\z@\hbox{$\macc@style{\macc@nucleus}_{}$}%
		\setbox\tw@\hbox{$\macc@style{\macc@nucleus}{}_{}$}%
		\dimen@\wd\tw@
		\advance\dimen@-\wd\z@
		\divide\dimen@ 3
		\@tempdima\wd\tw@
		\advance\@tempdima-\scriptspace
		\divide\@tempdima 10
		\advance\dimen@-\@tempdima
		\ifdim\dimen@>\z@ \dimen@0pt\fi
		\rel@kern{0.6}\kern-\dimen@
		\if#31
		\overline{\rel@kern{-0.6}\kern\dimen@\macc@nucleus\rel@kern{0.4}\kern\dimen@}%
		\advance\dimen@0.4\dimexpr\macc@kerna
		\let\final@kern#2%
		\ifdim\dimen@<\z@ \let\final@kern1\fi
		\if\final@kern1 \kern-\dimen@\fi
		\else
		\overline{\rel@kern{-0.6}\kern\dimen@#1}%
		\fi
	}%
	\macc@depth\@ne
	\let\math@bgroup\@empty \let\math@egroup\macc@set@skewchar
	\mathsurround\z@ \frozen@everymath{\mathgroup\macc@group\relax}%
	\macc@set@skewchar\relax
	\let\mathaccentV\macc@nested@a
	\if#31
	\macc@nested@a\relax111{#1}%
	\else
	\def\gobble@till@marker##1\endmarker{}%
	\futurelet\first@char\gobble@till@marker#1\endmarker
	\ifcat\noexpand\first@char A\else
	\def\first@char{}%
	\fi
	\macc@nested@a\relax111{\first@char}%
	\fi
	\endgroup
}
\makeatother

\newcommand{\talpha}{\widetilde{\alpha}}
\newcommand{\tbeta}{\widetilde{\beta}}

\newcommand{\PSh}{\mathrm{PSh}}

\newcommand{\Shv}{\mathrm{Shv}}

\newcommand{\infcat}{$\infty$-category\xspace}
\newcommand{\infcats}{$\infty$-categories\xspace}

\newcommand{\inftopos}{$\infty$-topos\xspace}

\newcommand{\tauet}{\tau_\mathrm{\acute{e}t}}

\newcommand{\Mod}{\mathrm{Mod}}

\newcommand{\Coh}{\mathrm{Coh}}
\newcommand{\Cohb}{\mathrm{Coh}^{\mathsf b}}
\newcommand{\Cohh}{\mathrm{Coh}^\heartsuit}

\newcommand{\An}{\mathrm{An}}

\newcommand{\Top}{\mathcal T\mathrm{op}}

\newcommand{\dfDM}{\mathrm{dfDM}}
\newcommand{\dfSch}{\mathrm{dfSch}}
\newcommand{\rig}{\mathrm{rig}}
\newcommand{\rigg}{(-)^{\mathrm{rig}}}
\newcommand{\loc}{\mathrm{loc}}

\newcommand{\cTad}{\cT_{\mathrm{adic}}(k^\circ)}

\newcommand{\dAn}{\mathrm{dAn}}

\newcommand{\dAnk}{\mathrm{dAn}_k}

\newcommand{\cTan}{\cT_{\mathrm{an}}}
\newcommand{\cTank}{\cT_{\mathrm{an}}(k)}
\newcommand{\cTdisc}{\cT_{\mathrm{disc}}}
\newcommand{\cTdisck}{\cT_{\mathrm{disc}}(k)}
\newcommand{\cTet}{\cT_{\mathrm{\acute{e}t}}}

\newcommand{\Str}{\mathrm{Str}}
\newcommand{\Strloc}{\mathrm{Str}^\mathrm{loc}}
\newcommand{\RTop}{\tensor*[^\rR]{\Top}{}}

\newcommand{\dAfd}{\mathrm{dAfd}}

\newcommand{\trunc}{\mathrm{t}_0}

\newcommand{\CAlg}{\mathrm{CAlg}}
\newcommand{\CAlgad}{\mathrm{CAlg}^{\mathrm{ad}}}

\newcommand{\Cat}{\mathrm{Cat}}

\newcommand{\fib}{\mathrm{fib}}
\newcommand{\DerAn}{\mathrm{Der}\an}
\newcommand{\anL}{\mathbb L\an}

\newcommand{\adL}{\mathbb L^{\mathrm{ad}}}

\newcommand{\dAff}{\mathrm{dAff}}

\newcommand{\bfMap}{\mathbf{Map}}

\newcommand{\PrL}{\mathcal P \mathrm{r}^{\mathrm{L}}}

\newcommand{\Perf}{\mathrm{Perf}}

\newcommand{\Hilb}{\mathrm{Hilb}}

\newcommand{\FormalModels}{\mathrm{FM}}
\newcommand{\Ind}{\mathrm{Ind}}

\newcommand{\cofib}{\mathrm{cofib}}

\newcommand{\kc}{k^\circ}

\newcommand{\an}{^\mathrm{an}}
\newcommand{\alg}{^\mathrm{alg}}

\newcommand{\et}{_\mathrm{\acute{e}t}}
\newcommand{\fet}{_\mathrm{f\acute{e}t}}

\newcommand{\inv}{^{-1}}

\newcommand{\kanal}{$k$-analytic\xspace}

\newcommand{\op}{^\mathrm{op}}

\newcommand{\DM}{Deligne-Mumford\xspace}

\newcommand*{\longhookrightarrow}{\ensuremath{\lhook\joinrel\relbar\joinrel\rightarrow}}

\usetikzlibrary{decorations.markings} 
\tikzset{
  closed/.style = {decoration = {markings, mark = at position 0.5 with { \node[transform shape, xscale = .8, yscale=.4] {/}; } }, postaction = {decorate} },
  open/.style = {decoration = {markings, mark = at position 0.5 with { \node[transform shape, scale = .7] {$\circ$}; } }, postaction = {decorate} }
}

\DeclareMathOperator{\Fun}{Fun}

\DeclareMathOperator{\Hom}{Hom}

\DeclareMathOperator{\Map}{Map}

\DeclareMathOperator{\Sp}{Sp}

\DeclareMathOperator{\Spec}{Spec}
\DeclareMathOperator{\Spf}{Spf}

\DeclareMathOperator*{\colim}{colim}

\begin{document}
\title{Derived non-archimedean analytic Hilbert space}

\author{Jorge ANT\'ONIO}
\address{Jorge Ant\'onio,  Institut de Math\'ematiques de Toulouse, 118 Rue de Narbonne  31400 Toulouse}
\email{\texttt{jorge\_tiago.ferrera\_antonio@math.univ-toulouse.fr}}

\author{Mauro PORTA}
\address{Mauro PORTA, Institut de Recherche Mathématique Avancée, 7 Rue René Descartes, 67000 Strasbourg, France}
\email{porta@math.unistra.fr}

\date{\today}

\begin{abstract}
In this short paper we combine the representability theorem introduced in \cite{Porta_Yu_Representability,Porta_Yu_Mapping} with the theory of derived formal models introduced in \cite{Antonio_Formal_models} to prove the existence representability of the derived Hilbert space $\mathbf{R} \mathrm{Hilb}(X)$ for a separated \kanal space $X$.
Such representability results relies on a localization theorem stating that if $\fX$ is a quasi-compact and quasi-separated formal scheme, then the \infcat $\mathrm{Coh}^+(\fX^\rig)$ of almost perfect complexes over the generic fiber can be realized as a Verdier quotient of the \infcat $\Coh^+(\fX)$.
Along the way, we prove several results concerning the the \infcats of formal models for almost perfect modules on derived $k$-analytic spaces.
\end{abstract}

\maketitle

\personal{PERSONAL COMMENTS ARE SHOWN!!!}

\tableofcontents

\section{Introduction}

Let $k$ be a non-archimedean field equipped with a non-trivial valuation of rank $1$.
We let $k^\circ$ denote its ring of integers, $\mathfrak m$ an ideal of definition.
We furthermore assume that $\mathfrak m$ is finitely generated.
Given a separated \kanal space $X$, we are concerned with the existence of the \emph{derived} moduli space $\mathbf{R} \mathrm{Hilb}(X)$, which parametrizes flat families of closed subschemes of $X$.
The truncation of $\mathbf{R} \mathrm{Hilb}(X)$ coincides with the classical Hilbert scheme functor, $\Hilb(X)$, which has been shown to be representable by a \kanal space in \cite{Conrad_Spreading-out}.
On the other hand, in algebraic geometry the representability of the derived Hilbert scheme is an easy consequence of Artin-Lurie representability theorem.
In this paper, we combine the analytic version of Lurie's representability obtained by T.\ Y.\ Yu and the second author in \cite{Porta_Yu_Representability} together with a theory of derived formal models developped by the first author in \cite{Antonio_Formal_models}.
The only missing step is to establish the existence of the cotangent complex.

Indeed, the techniques introduced in \cite{Porta_Yu_Mapping} allows to prove the existence of the cotangent complex at points $x \colon S \to \mathbf{R} \mathrm{Hilb}(X)$ corresponding to families of closed subschemes $j \colon Z \hookrightarrow S \times X$ which are of finite presentation in the derived sense.
However, not every point of $\mathbf{R} \mathrm{Hilb}(X)$ satisfies this condition: typically, we are concerned with families which are \emph{almost} of finite presentation.
The difference between the two situations is governed by the relative analytic cotangent complex $\anL_{Z / S \times X}$: $Z$ is (almost) of finite presentation if $\anL_{Z/S \times X}$ is (almost) perfect.
We can explain the main difficulty as follows: if $p \colon Z \to S$ denotes the projection to $S$, then the cotangent complex of $\mathbf{R} \mathrm{Hilb}(X)$ at $x \colon S \to \mathbf{R} \mathrm{Hilb}(X)$ is computed by $p_+( \anL_{Z / S \times X} )$.
Here, $p_+$ is a (partial) left adjoint for the functor $p^*$, which has been introduced in the \kanal setting in \cite{Porta_Yu_Mapping}.
However, in loc.\ cit.\ the functor $p_+$ has only been defined on perfect complexes, rather than on almost perfect complexes.
From this point of view, the main contribution of this paper is to provide an extension of the construction $p_+$ to almost perfect complexes.
Our construction relies heavily on the existence results for formal models of derived \kanal spaces obtained by the first author in \cite{Antonio_Formal_models}.
Along the way, we establish three results that we deem to be of independent interest, and which we briefly summarize below.\\

Let $\fX$ be a derived formal $\kc$-scheme topologically almost of finite presentation.
One of the main construction of \cite{Antonio_Moduli_of_representations, Antonio_Formal_models, antonio2019moduli} is the generic fiber $\fX^\rig$, which is a derived \kanal space.
The formalism introduced in loc.\ cit.\ provides as well an exact functor
\begin{equation} \label{eq:generic_fiber_coherent_sheaves}
	(-)^\rig \colon \Coh^+(\fX) \longrightarrow \Coh^+( \fX^\rig ) ,
\end{equation}
where $\Coh^+$ denotes the stable $\infty$-category of almost perfect complexes on $\fX$ and on $\fX^\rig$.
When $\fX$ is underived, this functor has been considered at length in \cite{Hennion_Porta_Vezzosi_Formal_gluing}, where in particular it has been shown to be essentially surjective, thereby extending the classical theory of formal models for coherent sheaves on \kanal spaces.
In this paper we extend this result to the case where $\fX$ is derived, which is a key technical step in our construction of the plus pushforward. 
In order to do so, we will establish the following descent statement, which is an extension of \cite[Theorem 7.3]{Hennion_Porta_Vezzosi_Formal_gluing}:

\begin{thm-intro} \label{thm:intro_descent}
	The functor $\mathrm{Coh}^+_{\mathrm{loc}} \colon \dAnk \to \mathrm{Cat}_\infty^{\mathrm{st}}$, which associates to every derived formal derived scheme
		\[
				\fX \in \dfDM \mapsto \mathrm{Coh}^+(\fX^\rig)
				 \in \mathrm{Cat}^\mathrm{st}_{\infty},
		\]	
	satisfies Zariski hyper-descent.
\end{thm-intro}

We refer the reader to \cref{thm:descent_punctured_category} for the precise statement. S consequence of \cref{thm:intro_descent} above is the following statement, concerning the properties of \infcats of formal models for almost perfect complexes on $X \in \dAnk$:

\begin{thm-intro}[\cref{thm:formal_models_filtered}] \label{thm:intro_cat_formal_model}
	Let $X \in \dAn_k$ be a derived $k$-analytic space and let $\cF \in \Coh^+(\cF)$ be a bounded below almost perfect complex on $X$.
	For any derived formal model $\fX$ of $X$, there exists $\cG \in \Coh^+(\fX)$ and an equivalence $\cG^\rig \simeq \cF$.
	Furthermore, the full subcategory of $\Coh^+(\fX) \times_{\Coh^+(X)} \Coh^+(X)_{/\cF}$ spanned by formal models of $\cF$ is filtered.
\end{thm-intro}

\cref{thm:intro_cat_formal_model} is another key technical ingredient in the proof of the existence of a plus pushforward construction.
The third auxiliary result we need is a refinement of the existence theorem for formal models for morphisms of derived analytic spaces proven in \cite{Antonio_Formal_models}.
It can be stated as follows:

\begin{thm-intro}[\cref{thm:formal_model_flat_map}] \label{thm:intro_flat_lifting}
	Let $f \colon X \to Y$ be a flat map between derived \kanal spaces.
	Then there are formal models $\fX$ and $\fY$ for $X$ and $Y$ respectively and a flat map $\ff \colon \fX \to \fY$ whose generic fiber is equivalent to $f$.
\end{thm-intro}

The classical analogue of \cref{thm:intro_flat_lifting} was proven by Bosch and Lutk\"ebohmert in \cite{Bosch_Formal_IV}.
The proof of this theorem is not entirely obvious: indeed the algorithm provided in \cite{Antonio_Formal_models} proceeds by induction on the Postnikov tower of both $X$ and $Y$, and at each step uses \cite[Theorem 7.3]{Hennion_Porta_Vezzosi_Formal_gluing} to choose appropriately formal models for $\pi_i(\cO_X\alg)$ and $\pi_i(\cO_Y\alg)$.
In the current situation, however, the flatness requirement on $\ff$ makes it impossible to freely choose a formal model for $\pi_i(\cO_X\alg)$.
We circumvent the problem by proving a certain lifting property for morphisms of almost perfect complexes:

\begin{thm-intro}[\cref{lift_p_maps}]
	Let $X \in \dAn_k$ be a derived \kanal space and let $f \colon \cF \to \cG$ be a morphism in $\Coh^+(X)$.
	Let $\fX$ denote a given formal model for $X$. Suppose, futhermore, that we are given formal models $\widetilde{\cF}, \widetilde{\cG} \in \Coh^+(X)$ for $\cF$ and $\cG$, respectively. Then, there exists a non-zero element $t \in \mathfrak m$ such that the map $t^n f$ admits a lift $\widetilde{f} \colon \widetilde{\cF} \to \widetilde{\cG}$,
	in the \infcat $\mathrm{Coh}^+(\fX)$.
\end{thm-intro}

Finally, the techniques of the current text allow us to prove the following generalization of \cite[Theorem 8.6]{Porta_Yu_Mapping}:

\begin{thm-intro}[\cref{thm:rep_mapping}]
	Let $S$ be a rigid \kanal space.
	Let $X,Y$ be rigid \kanal spaces over $S$.
	Assume that $X$ is proper and flat over $S$ and that $Y$ is separated over $S$.
	Then the $\infty$-functor $\bfMap_S(X,Y)$ is representable by a derived \kanal space separated over $S$.
\end{thm-intro}

\paragraph{\bf{Notation and conventions}}

\todo{Modify: not all notations below are useful, and we need others. For instance, $\mathrm{fSch}_{\kc}$, $\bA^1_{\kc} \coloneqq \Spf( \kc \{ T \} )$...}

In this paper we freely use the language of $\infty$-categories.
Although the discussion is often independent of the chosen model for $\infty$-categories, whenever needed we identify them with quasi-categories and refer to \cite{HTT} for the necessary foundational material.

The notations $\cS$ and $\Cat_\infty$ are reserved to denote the $\infty$-categories of spaces and of $\infty$-categories, respectively.
If $\cC \in \Cat_\infty$ we denote by $\cC^\simeq$ the maximal $\infty$-groupoid contained in $\cC$.
We let $\Cat_\infty^{\mathrm{st}}$ denote the $\infty$-category of stable $\infty$-categories with exact functors between them.
We also let $\PrL$ denote the $\infty$-category of presentable $\infty$-categories with left adjoints between them.
Similarly, we let $\PrL_{\mathrm{st}}$ denote the $\infty$-categories of stably presentable $\infty$-categories with left adjoints between them.
Finally, we set
\[ \Cat_\infty^{\mathrm{st}, \otimes} \coloneqq \CAlg( \Cat_\infty^{\mathrm{st}} ) \quad , \quad \mathcal P \mathrm r_{\mathrm{st}}^{\mathrm L, \otimes} \coloneqq \CAlg( \PrL_{\mathrm{st}} ) . \]

Given an $\infty$-category $\cC$ we denote by $\PSh(\cC)$ the $\infty$-category of $\cS$-valued presheaves.
We follow the conventions introduced in \cite[\S 2.4]{Porta_Yu_Higher_analytic_stacks_2014} for $\infty$-categories of sheaves on an $\infty$-site.

For a field $k$, we reserve the notation $\CAlg_k$ for the $\infty$-category of simplicial commutative rings over $k$.
We often refer to objects in $\CAlg_k$ simply as \emph{derived commutative rings}.
We denote its opposite by $\dAff_k$, and we refer to it as the $\infty$-category of \emph{derived affine schemes}.
We say that a derived ring $A \in \CAlg_k$ is \emph{almost of finite presentation} if $\pi_0(A)$ is of finite presentation over $k$ and $\pi_i(A)$ is a finitely presented $\pi_0(A)$-module.\footnote{Equivalently, $A$ is almost of finite presentation if $\pi_0(A)$ is of finite presentation and the cotangent complex $\mathbb L_{A/k}$ is an almost perfect complex over $A$.}
We denote by $\dAff_k^{\mathrm{afp}}$ the full subcategory of $\dAff_k$ spanned by derived affine schemes $\Spec(A)$ such that $A$ is almost of finite presentation.
When $k$ is either a non-archimedean field equipped with a non-trivial valuation or is the field of complex numbers, we let $\An_k$ denote the category of analytic spaces over $k$.
We denote by $\Sp(k)$ the analytic space associated to $k$.

\paragraph{\textbf{Acknowledgments}}

We are grateful to B.\ Hennion, M.\ Robalo, B.\ To\"en, G.\ Vezzosi and T.\ Y.\ Yu for useful discussions related to the content of this paper.
\todo{Other people?}

We are especially grateful to B.\ To\"en for inviting the second author to Toulouse during the fall 2017, when the bulk of this research has been carried out and to IRMA, Universit\'e de Strasbourg, for inviting the first author to Strasbourg during the spring 2018. 
\todo{Acknowledge the PEPS.}

\section{Preliminaries on derived formal and derived non-archimedean geometries}

Let $k$ denote a non-archimedean field equipped with a rank $1$ valuation.
We let $\kc = \{ x \in k : | x | \leq 1 \}$ denote its ring of integers.
We assume that $\kc$ admits a finitely generated ideal of definition $\mathfrak m$.

\begin{notation} \label{notation:pregeometries}
	\begin{enumerate}
		\item Let $R$ be a discrete commutative ring. Let $\cTdisc(R)$ denote the full subcategory of $R$-schemes spanned by affine spaces $\mathbb A^n_R$. We say that a morphism in $\cTdisc(R)$ is \emph{admissible} if it is an isomorphism.
		We endow $\cTdisc(R)$ with the trivial Grothendieck topology.
		
		\item Let $\cTad$ denote the full subcategory of $\kc$-schemes spanned by formally smooth formal schemes which are topologically finitely generated over $\kc$. A morphism in $\cTad$ is said to be \emph{admissible} if it is formally \'etale.
		We equip the category $\cTad$ with the formally \'etale topology, $\tau\et$.
		
		\item Denote $\cTan(k)$ the category of smooth $k$-analytic spaces. A morphism in $\cTan(k)$ is said to be \emph{admissible} if it is \'etale. We endow $\cTank$ with the \'etale topology, $\tau\et$. 
	\end{enumerate}
\end{notation}

In what follows, we will let $\cT$ denote either one of the categories introduced above.
We let $\tau$ denote the corresponding Grothendieck topology.

\begin{defin}
Let $\cX$ be an \inftopos. A \emph{$\cT$-structure} on $\cX$ is a functor $\cO \colon \cT \to \cX$ which commutes with finite products, pullbacks along admissible morphisms and takes $\tau$-coverings in effective epimorphisms. We denote by $\Str_{\cT}( \cX )$ the full subcategory of $\Fun_{\cT} \left( \cT, \cX \right)$ spanned by $\cT$-structures.
A \emph{$\cT$-structured $\infty$-topos} is a pair $(\cX, \cO)$, where $\cX$ is an $\infty$-topos and $\cO \in \Str_\cT(\cX)$.
\end{defin}

We can assemble $\cT$-structured $\infty$-topoi into an $\infty$-category denoted $\RTop(\cT)$.
We refer to \cite[Definition 1.4.8]{DAG-V} for the precise construction.
\personal{Let $\RTop$ be the \infcat of $\infty$-topoi together with right geometric morphisms: a morphism $f \colon \cX \to \cY$ is a geometric morphism $f_* \colon \cX \rightleftarrows \cY \colon f\inv$ where $f_*$ is the right adjoint.
The functor $\Fun(\cT,-) \colon \Cat_\infty \to \Cat_\infty$ restricts to a functor
\[ \Fun(\cT,-) \colon \left(\RTop\right)\op \longrightarrow \Cat_\infty , \]
which sends a geometric morphism $(f\inv, f_*)$ to the functor induced by composition with $f\inv$.
Since the left adjoint of a geometric morphism preserves finite limits, it follows that it respects the full subcategories of $\cT$-structures.
In other words, we obtain a well defined functor
\[ \Str_\cT \colon \left( \RTop \right)\op \longrightarrow \Cat_\infty . \]
This defines a Cartesian fibration $p_{\Str} \colon \RTop( \cT)  \to \RTop\op$ and we can identify objects of $\RTop$ as pairs $(\cX, \cO)$, where $\cX \in \RTop$ and $\cO \in \Str_{\cT} ( \cX)$. We say that an object of $\RTop ( \cT)$ is a \emph{$\cT$-structured \inftopos}.}

\begin{defin}
	Let $\cX$ be an $\infty$-topos.
	A morphism of $\cT$-structures $\alpha \colon \cO \to \cO'$ is said to be \emph{local} if for every admissible morphism $f \colon U \to V$ in $\cT$ the diagram
	\[ \begin{tikzcd}
	\cO (U) \ar{r}{\cO(f)} \ar{d}{\alpha_U} & \cO(V) \ar{d}{\alpha_V} \\
	\cO'(U) \ar{r}{\cO'(f)} & \cO'(V)
	\end{tikzcd} \]
	is a pullback square in $\cX$.
	We denote by $\Strloc_{\cT} ( \cX )$ the (non full) subcategory of $\Str_{\cT} \left( \cX \right)$ spanned by local structures and local morphisms between these.
\end{defin}

\begin{eg} \label{eg1}
\begin{enumerate}
	\item Let $R$ be a discrete commutative ring. A $\cTdisc(R)$-structure on an $\infty$-topos $\cX$ is simply a product preserving functor $\cO \colon \cTdisc(R) \to \cX$.
	When $\cX = \cS$ is the $\infty$-topos of spaces, we can therefore use \cite[Proposition 5.5.9.2]{HTT} to identify the $\infty$-category $\Str_{\cTdisc(R)}(\cX)$ with the underlying $\infty$-category $\CAlg_R$ of the model category of simplicial commutative $R$-algebras.
	It follows that $\Str_{\cTdisc(R)}(\cX)$ is canonically identified with the $\infty$-category of sheaves on $\cX$ with values in $\CAlg_R$.
	For this reason, we write $\CAlg_R(\cX)$ rather than $\Strloc_{\cTdisc(R)}(\cX)$.
	
	\item Let $\fX$ denote a formal scheme over $\kc$ complete along $t \in \kc$. Denote by $\fX\fet$ the small formal \'etale site on $\fX$ and denote $\cX := \Shv(\fX\fet, \tau\et)^\wedge$ denote the hypercompletion of the $\infty$-topos of formally \'etale sheaves on $\fX$. We define a $\cTad$-structure on $\cX$ as the functor which sends $U \in \fX\fet$ to the sheaf $\cO( U) \in \cX$ defined by the association
		\[ V \in \fX\fet \mapsto \Hom_{\mathrm{fSch}_{\kc}} \left( V, U \right) \in \cS. \]
	In this case, $\cO( \bA^1_{\kc})$ corresponds to the sheaf of functions on $\fX$ whose support is contained in the $t$-locus of $\fX$.
	\personal{We can as well perform a similar construction in the case where $\fX$ is Deligne-Mumford stack over $\kc$, complete along $t \in \kc$.}
	To simplify the notation, we write $\mathrm{fCAlg}_{\kc}(\cX)$ rather than $\Strloc_{\cTad}(\cX)$.
	
	\item Let $X$ be a \kanal space and denote $X\et$ the associated small \'etale site on $X$. Let $\cX := \Shv(X\et, \tauet)^{\wedge}$ denote the hypercompletion of the $\infty$-topos of \'etale sheaves on $X$. We can attach to $X$ a $\cTank$-structure on $\cX$ as follows: given $U \in \cTank$, we define the sheaf $\cO(U) \in \cX$ by
	\[ X\et \ni V \mapsto \Hom_{\An_k} \left( V, U \right) \in \cS . \]
	As in the previous case, we can canonically identify $\cO(\bA^1_k)$ with the usual sheaf of analytic functions on $X$.
	We write $\mathrm{AnRing}_k(\cX)$ rather than $\Strloc_{\cTan(k)}(\cX)$.
\end{enumerate}
\end{eg}

\begin{construction} \label{construction:morphims_pregeometries}
	Let $\cX$ be an $\infty$-topos.
	We can relate the $\infty$-categories $\Str_{\cTdisc(\kc)}(\cX)$, $\Str_{\cTdisck}(\cX)$, $\Str_{\cTad}(\cX)$ and $\Str_{\cTank}(\cX)$ as follows.
	Consider the following functors
	\begin{enumerate}
		\item the functor
		\[ - \otimes_{\kc} k \colon \cTdisc(\kc) \longrightarrow \cTdisck . \]
		induced by base change along the map $\kc \to k$.
		
		\item The functor
		\[ (-)^\wedge_t \colon \cTdisc(\kc) \longrightarrow \cTad . \]
		 induced by the $(t)$-completion.
		
		\item The functor
		\[ (-)\an \colon \cTdisck \longrightarrow \cTank , \]
		induced by the analytification.
		
		\item The functor
		\[ (-)^\rig \colon \cTad \longrightarrow \cTank \]
		induced by Raynaud's generic fiber construction (cf.\ \cite[Theorem 8.4.3]{Bosch_Lectures_2014}).
	\end{enumerate}
	These functors respect the classes of admissible morphisms and are continuous morphisms of sites.
	It follows that precomposition with them induce well defined functors
	\begin{gather*}
		\Str_{\cTdisck}(\cX) \longrightarrow \Str_{\cTdisc(\kc)}(\cX) \quad , \quad (-)\alg \colon \Str_{\cTad}(\cX) \longrightarrow \Str_{\cTdisc(\kc)}(\cX) \\
		(-)^+ \colon \Str_{\cTank}(\cX) \longrightarrow \Str_{\cTad}(\cX) \quad , \quad (-)\alg \colon \Str_{\cTank}(\cX) \longrightarrow \Str_{\cTdisck}(\cX) .
	\end{gather*}
	The first functor simply forgets the $k$-algebra structure to a $\kc$-algebra one via the natural map $\kc \to k$.
	We refer to the second and fourth functors as the \emph{underlying algebra functors}.
	The third functor is an analogue of taking the subring of power-bounded elements in rigid geometry.
\end{construction}

Using the underlying algebra functors introduced in the above construction, we can at last introduce the definitions of derived formal scheme and derived \kanal space.
They are analogous to each other:

\begin{defin}
	A $\cTad$-structured $\infty$-topos $\fX \coloneqq (\cX, \cO_\fX)$ is said to be a \emph{derived formal \DM $\kc$-stack} if there exists a collection of objects $\{U_i\}_{i \in I}$ in $\cX$ such that $\coprod_{i \in I} U_i \to \mathbf 1_\cX$ is an effective epimorphism and the following conditions are met:
	\begin{enumerate}
		\item for every $i \in I$, the $\cTad$-structured $\infty$-topos $(\cX_{/U_i}, \pi_0(\cO_\fX |_{U_i}))$ is equivalent to the $\cTad$-structured $\infty$-topos arising from an affine formal $\kc$-scheme via the construction given in \cref{eg1}.
		
		\item For each $i \in I$ and each integer $n \ge 0$, the sheaf $\pi_n( \cO_\fX\alg |_{U_i} )$ is a quasi-coherent sheaf over $(\cX_{/U_i}, \pi_0( \cO_\fX |_{U_i} ))$. 
	\end{enumerate}
	We say that $\fX = (\cX, \cO_\fX)$ is a \emph{formal derived $\kc$-scheme} if it is a derived formal Deligne Mumford stack and furthermore its truncation $\trunc(\fX) \coloneqq (\cX, \pi_0(\cO_\fX))$ is equivalent to the $\cTad$-structured $\infty$-topos associated to a formal scheme via \cref{eg1}.
\end{defin}

\begin{defin}
	A $\cTank$-structured $\infty$-topos $X \coloneqq (\cX, \cO_X)$ is said to be a \emph{derived \kanal space} if $\cX$ is hypercomplete and there exists a collection of objects $\{U_i\}_{i \in I}$ in $\cX$ such that $\coprod_{i \in I} U_i \to \mathbf 1_\cX$ is an effective epimorphism and the following conditions are met:
	\begin{enumerate}
		\item for each $i \in I$, the $\cTank$-structured $\infty$-topos $( \cX_{ / U_i}, \pi_0( \cO_X |_{U_i} ) )$ is equivalent to the $\cTank$-structured $\infty$-topos arising from an ordinary \kanal space via the construction given in \cref{eg1}.
		\item For each $i \in I$ and each integer $n \geq 0$, the sheaf $ \pi_n(\cO_X\alg |_{U_i})$ is a coherent sheaf on $(\cX_{/U_i}, \cO_X |_{U_i})$.
	\end{enumerate}
\end{defin}

\begin{thm}[{cf.\ \cite{Antonio_Formal_models,DAG-IX,Porta_Yu_DNAnG_I} }]
	Derived formal \DM $\kc$-stacks and derived \kanal spaces assemble into $\infty$-categories, denoted respectively $\dfDM_{\kc}$ and $\dAn_k$, which enjoy the following properties:
	\begin{enumerate}
		\item fiber products exist in both $\dfDM_{\kc}$ and $\dAnk$;
		\item The constructions given in \cref{eg1} induce full faithful embeddings from the categories of ordinary formal \DM $\kc$-stacks $\mathrm{fDM}_{\kc}$ and of ordinary \kanal spaces $\An_k$ in $\dfDM_{\kc}$ and $\dAn_k$, respectively.
	\end{enumerate}
\end{thm}

Following \cite[\S 8.1]{Lurie_SAG}, we let $\CAlgad$ denote the $\infty$-category of simplicial commutative rings equipped with an adic topology on their $0$-th truncation.
Morphisms are morphisms of simplicial commutative rings that are furthermore continuous for the adic topologies on their $0$-th truncations.
We set
\[ \CAlgad_{\kc} \coloneqq \CAlgad_{\kc/} ,  \]
where we regard $\kc$ equipped with its $\mathfrak m$-adic topology.
Thanks to \cite[Remark 3.1.4]{Antonio_Formal_models}, the underlying algebra functor $(-)\alg \colon \mathrm{fCAlg}_{\kc}(\cX) \to \CAlg_{\kc}(\cX)$ factors through $\CAlgad_{\kc}(\cX)$.
\personal{Fix $A \in \mathrm{fCAlg}_{\kc}(\cX)$. Then we are \emph{not} equipping $A\alg$ with the $(t)$-adic topology (although this would indeed give a functor $\mathrm{fCAlg}_{\kc}(\cX) \to \CAlgad_{\kc}(\cX)$, except that it would be induced by $\CAlg_{\kc}(\cX) \to \CAlgad_{\kc}(\cX)$, which is not what we want).
	To understand this functor, consider the reduction $\kc_n$ of $\kc$ modulo $(t^n)$.
	Then for every $n \ge 1$ we have $\kc \to \kc_n$, which induces a transformation of pregeometries $\cTad \to \cTet(\kc_n)$, and therefore a functor
	\[ L_n \colon \mathrm{fCAlg}_{\kc}(\cX) \longrightarrow \CAlg_{\kc_n}(\cX) , \]
	which satisfies
	\[ L_n(A) \simeq A\alg \otimes_{\kc} \kc_n . \]
	For every $n$, consider $I_n \coloneqq \ker( \pi_0( A\alg ) \to \pi_0( A\alg \otimes_{\kc} \kc_n ) )$.
	Then the sequence $\{I_n\}$ defines an adic topology on $A\alg$, which is compatible with the $(t)$-adic topology on $\kc$.
	It follows that $(A\alg, \{I_n\})$ defines an element in $\CAlgad_{\kc}(\cX)$.}
We denote by $(-)^{\mathrm{ad}}$ the resulting functor:
\[ (-)^{\mathrm{ad}} \colon \mathrm{fCAlg}_{\kc}(\cX) \longrightarrow \CAlgad_{\kc}(\cX) . \]

\begin{defin}
	Let $A \in \mathrm{fCAlg}_{\kc}(\cX)$.
	We say that $A$ is \emph{topologically almost of finite type over $\kc$} if the underlying sheaf of $\kc$-adic algebras $A^{\mathrm{ad}}$ is $t$-complete, $\pi_0(A\alg)$ is sheaf of topologically of finite type $\kc$-adic algebras and for each $i > 0$, $\pi_i(A)$ is finitely generated as $\pi_0(A)$-module.
	
	We say that a derived formal \DM stack $\fX \coloneqq (\cX, \cO_\fX)$ if \emph{topologically almost of finite type over $\kc$} if its underlying $\infty$-topos is coherent (cf.\ \cite[\S 3]{DAG-VII}) and $\cO_\fX \in \mathrm{fCAlg}_{\kc}(\cX)$ is topologically almost of finite type over $\kc$.
	We denote by $\dfDM^{\mathrm{taft}}$  (resp.\ $\mathrm{dfSch}^{\mathrm{taft}}$) the full subcategory of $\dfDM_{\kc}$ spanned by those derived formal \DM stacks $\fX$ that are topologically almost of finite type over $\kc$ (resp.\ and whose truncation $\trunc(\fX)$ is equivalent to a formal $\kc$-scheme).
\end{defin}

The transformation of pregeometries
\[ \rigg \colon \cTad \longrightarrow \cTank \]
induced by Raynaud's generic fiber functor induces $\RTop(\cTank) \to \RTop(\cTad)$.
\cite[Theorem 2.1.1]{DAG-V} provides a right adjoint to this last functor, which we still denote
\[ \rigg \colon \RTop(\cTad) \longrightarrow \RTop(\cTank) . \]
We refer to this functor as the \emph{derived generic fiber functor} or as the \emph{derived rigidification functor}.

\begin{thm}[{\cite[Corollary 4.1.4, Proposition 4.1.6]{Antonio_Formal_models}}] \label{prop21}
	The functor $\rigg \colon \RTop(\cTad) \to \RTop(\cTank)$ enjoys the following properties:
	\begin{enumerate}
		\item it restricts to a functor
		\[ \rigg \colon \dfDM^{\mathrm{taft}} \longrightarrow \dAn_k . \]
		
		\item The restriction of $\rigg \colon \dfDM^{\mathrm{taft}} \to \dAn_k$ to the full subcategory $\mathrm{fSch}_{\kc}^{\mathrm{taft}}$ is canonically equivalent to Raynaud's generic fiber functor.
		
		\item Every derived analytic space $X \in \dAn_k$ whose truncation is an ordinary \kanal space\footnote{The $\infty$-category $\dAnk$ also contains \kanal \DM stacks.} lies in the essential image of the functor $\rigg$.
	\end{enumerate}
\end{thm}

Fix a derived formal \DM stack $\fX \coloneqq (\cX, \cO_{\fX})$ and a derived \kanal space $Y \coloneqq (\cY, \cO_Y)$.
We set
\[ \cO_{\fX} \textrm{-} \Mod \coloneqq \cO_{\fX}\alg \textrm{-} \Mod \quad , \quad \cO_Y \textrm{-} \Mod \coloneqq \cO_Y\alg \textrm{-} \Mod . \]
We refer to $\cO_{\fX} \textrm{-} \Mod$ as the \emph{stable $\infty$-category of $\cO_{\fX}$-modules}.
Similarly, we refer to $\cO_Y \textrm{-} \Mod$ as the \emph{stable $\infty$-category of $\cO_Y$-modules}.
The derived generic fiber functor induces a functor
\[ \rigg \colon \cO_{\fX} \textrm{-} \Mod \longrightarrow \cO_{\fX^\rig} \textrm{-} \Mod . \]

\begin{defin}
	Let $\fX \in \dfDM_{\kc}$ be a derived $\kc$-adic \DM stack and let $X \in \dAnk$ be a derived \kanal space.
	The $\infty$-category $\Coh^+(\fX)$ (resp.\ $\Coh^+(X)$) of almost perfect complexes on $\fX$ (resp.\ on $X$) is the full subcategory of $\cO_{\fX} \textrm{-} \Mod$ (resp.\ of $\cO_X \textrm{-} \Mod$) spanned by those $\cO_{\fX}$-modules (resp.\ $\cO_X$-modules) $\cF$ such that $\pi_i( \cF )$ is a coherent sheaf on $\trunc(\fX)$ (resp.\ on $\trunc(X)$) for every $i \in \mathbb Z$ and $\pi_i(\cF) \simeq 0$ for $i \ll 0$.
\end{defin}

For later use, let us record the following result:

\begin{prop}[{\cite{Lurie_SAG} \& \cite[Theorem 3.4]{Porta_Yu_Mapping}}] \label{prop:modules_on_affines_and_affinoid}
	Let $\fX$ be a derived affine $\kc$-adic scheme.
	Let $A \coloneqq \Gamma( \fX; \cO_\fX\alg)$.
	Then the functor $\Gamma(\fX;-)$ restricts to
	\[ \Coh^+( \fX ) \longrightarrow \Coh^+( A ) \]
	and furthermore this is an equivalence.
	Similarly, if $X$ is a derived $k$-affinoid space,\footnote{By definition, $X$ is a derived $k$-affinoid space if $\trunc(X)$ is a $k$-affinoid space.} and $B \coloneqq \Gamma(X; \cO_X\alg)$, then $\Gamma(X;-)$ restricts to
	\[ \Coh^+( X ) \longrightarrow \Coh^+( B ) , \]
	and furthermore this is an equivalence.
\end{prop}

To complete this short review, we briefly discuss the notion of the $\kc$-adic and \kanal cotangent complexes.
The two theories are parallel, and for sake of brevity we limit ourselves to the first one.
We refer to the introduction of \cite{Porta_Yu_Representability} for a more thorough review of the \kanal theory.

In \cite[\S 3.4]{Antonio_Formal_models} it was constructed a functor
\[ \Omega^\infty_{\mathrm{ad}} \colon \cO_{\fX} \textrm{-} \Mod \longrightarrow \mathrm{fCAlg}_{\kc}(\cX)_{/\cO_\fX} , \]
which we refer to as the \emph{$\kc$-adic split square-zero extension functor}.
Given $\cF \in \cO_{\fX} \textrm{-} \Mod$, we often write $\cO_{\fX} \oplus \cF$ instead of $\Omega^\infty_{\mathrm{ad}}(\cF)$.

\begin{rem}
	Although the $\infty$-category $\cO_{\fX} \textrm{-} \Mod$ is \emph{not} sensitive to the $\cTad$-structure on $\cO_{\fX}$, the functor $\Omega^\infty_{\mathrm{ad}}$ depends on it in an essential way.
\end{rem}

\begin{defin}
	The functor of \emph{$\kc$-adic derivations} is the functor
	\[ \mathrm{Der}^{\mathrm{ad}}_{\kc}(\fX;-) \colon \cO_{\fX} \textrm{-} \Mod \longrightarrow \cS \]
	defined by
	\[ \mathrm{Der}^{\mathrm{ad}}_{\kc}(\fX;\cF) \coloneqq \Map_{\mathrm{fCAlg}_{\kc}(\cX)_{/\cO_\fX}}( \cO_\fX, \cO_\fX \oplus \cF ) . \]
\end{defin}

For formal reasons, the functor $\mathrm{Der}^{\mathrm{ad}}_{\kc}(\fX;-)$ is corepresentable by an object $\adL_\fX \in \cO_{\fX} \textrm{-} \Mod$.
We refer to it as the \emph{$\kc$-adic cotangent complex of $\fX$}.
The following theorem summarizes its main properties:

\begin{thm}[{\cite[Proposition 3.4.4, Corollary 4.3.5, Proposition 3.5.8]{Antonio_Formal_models}}] \label{thm:adic_and_analytic_cotangent_complex}
	Let $\fX \coloneqq (\cX, \cO_\fX)$ be a derived $\kc$-adic \DM stack.
	Let $\mathrm t_{\le n} \fX \coloneqq (\cX, \tau_{\le n} \cO_\fX)$ be the $n$-th truncation of $\fX$.
	Then:
	\begin{enumerate}
		\item the $\kc$-adic cotangent complex $\adL_\fX$ belongs to $\Coh^+(\fX)$;
		
		\item in $\Coh^+( \fX^\rig )$ there is a canonical equivalence
		\[ ( \adL_\fX )^\rig \simeq \anL_{\fX^\rig} , \]
		where $\anL_{\fX^\rig}$ denotes the analytic cotangent complex of the derived \kanal space $\fX^\rig$;
		
		\item the algebraic derivation classifying canonical map $(\fX, \tau_{\le n+1} \cO_\fX) \to (\fX, \tau_{\le n} \cO_\fX)$ can be canonically lifted to a $\kc$-adic derivation
		\[ \adL_{\mathrm t_{\le n} \fX} \longrightarrow \pi_{n+1}( \cO_\fX )[n+2] . \]
	\end{enumerate}
\end{thm}

\section{Formal models for almost perfect complexes}

\subsection{Formal descent statements}

We assume that $\kc$ admits a finitely generated ideal of definition $\mathfrak m$.
We also fix a set of generators $t_1, \ldots, t_n \in \mathfrak m$.
We start by recalling the notion of $\mathfrak m$-nilpotent almost perfect complexes.

\begin{defin}
	Let $\fX$ be a derived $\kc$-adic \DM stack topologically almost of finite presentation.
	We let $\Coh^+_{\mathrm{nil}}( \fX )$ denote the fiber of the generic fiber functor \eqref{eq:generic_fiber_coherent_sheaves}:
	\[ \Coh^+_{\mathrm{nil}}( \fX ) \coloneqq \fib\left( \Coh^+( \fX ) \xrightarrow{\rigg} \Coh^+( \fX^\rig )  \right) . \]
	We refer to $\Coh^+_{\mathrm{nil}}(\fX)$ as the full subcategory of $\mathfrak m$-nilpotent almost perfect complexes on $X$.
\end{defin}

A morphism $\ff \colon \fX \to \fY$ in $\dfDM_{\kc}^{\mathrm{taft}}$ induces a commutative diagram
\begin{equation} \label{eq:naturality_rigg}
	\begin{tikzcd}
		\Coh^+(\fY) \arrow{r}{\ff^*} \arrow{d}{\rigg} & \Coh^+( \fX ) \arrow{d}{\rigg} \\
		\Coh^+( \fY^\rig ) \arrow{r}{(\ff^\rig)^*} & \Coh^+( \fX^\rig ) .
	\end{tikzcd}
\end{equation}
In particular, we see that $\ff^*$ preserves the subcategory of $\mathfrak m$-nilpotent almost perfect complexes on $X$.
Moreover, as both $\Coh^+( \fX )$ and $\Coh^+( \fX^\rig )$ satisfy \'etale descent, we conclude that $\Coh^+_{\mathrm{nil}}(\fX)$ satisfies \'etale descent as well.

\begin{lem} \label{lem:characterization_nilpotent_almost_perfect_modules}
	Let $\fX$ be a derived $\kc$-adic \DM stack.
	Then an almost perfect sheaf $\cF \in \Coh^+(X)$ is $\mathfrak m$-nilpotent if and only if for every $i \in \mathbb Z$ the coherent sheaf $\pi_i(\cF)$ is annihilated by some power of the ideal $\mathfrak m$.
\end{lem}

\begin{proof}
	The question is \'etale local on $\fX$.
	In particular, we can assume that $\fX$ is a derived formal affine scheme topologically of finite presentation.
	Write
	\[ A \coloneqq \Gamma( \fX, \cO_\fX\alg ) . \]
	Let $X \coloneqq \fX^\rig$.
	Then \cite[Corollary 4.1.3]{Antonio_Formal_models} shows that
	\[ \trunc(\fX^\rig) \simeq ( \trunc(\fX) )^\rig . \]
	In particular, we deduce that $X$ is a derived $k$-affinoid space.
	Write
	\[ B \coloneqq \Gamma( X, \cO_X\alg ) . \]
	We can therefore use \cref{prop:modules_on_affines_and_affinoid} to obtain canonical equivalences
	\[ \Coh^+( \fX ) \simeq \Coh^+( A\alg ) \quad , \quad \Coh^+( X ) \simeq \Coh^+( B ) . \]
	Under these identifications, the functor $\rigg$ becomes equivalent to the base change functor
	\[ - \otimes_A B \colon \Coh^+(A) \longrightarrow \Coh^+(B) . \]
	Moreover, it follows from \cite[Proposition A.1.4]{Antonio_Formal_models} that there is a canonical identification
	\[ B \simeq A \otimes_{\kc} k . \]
	In particular, $\rigg \colon \Coh^+( \fX ) \to \Coh^+( X )$ is $t$-exact.
	The conclusion is now straightforward.
\end{proof}

\begin{defin}
	Let $\fX$ be a derived $\kc$-adic \DM stack.
	Let $\cF \in \Coh^+(\fX^\rig)$.
	We say that $\mathfrak F \in \Coh^+(\fX)$ is a formal model for $\cF$ if there exists an equivalence $\fF^\rig \simeq \cF$ in $\Coh^+(\fX^\rig)$.
	We let $\FormalModels(\cF)$ denote the full subcategory of
	\[ \Coh^+(\fX)_{/\cF} \coloneqq \Coh^+( \fX ) \times_{\Coh^+( \fX^\rig )} \Coh^+( \fX^\rig )_{/\cF} \]
	spanned by formal models of $\cF$.
\end{defin}

Our goal in this section is to study the structure of $\FormalModels(\cF)$, and in particular to establish that it is non-empty and filtered when $\fX$ is a quasi-compact and quasi-separated derived $\kc$-adic scheme.
Notice that saying that $\FormalModels(\cF)$ is non-empty for every choice of $\cF \in \Coh^+(X)$ is equivalent to asserting that the functor \eqref{eq:generic_fiber_coherent_sheaves}
\[ \rigg \colon \Coh^+( \fX ) \longrightarrow \Coh^+( X ) \]
is essentially surjective.

\begin{lem} \label{lem:existence_formal_models_affine}
	If $\fX$ is a derived $\kc$-affine scheme topologically almost of finite presentation, then the functor \eqref{eq:generic_fiber_coherent_sheaves} is essentially surjective.
\end{lem}

\begin{proof}
	We let
	\[ A \coloneqq \Gamma( \fX, \cO_\fX\alg ) \quad , \quad B \coloneqq \Gamma( \fX^\rig, \cO_{\fX^\rig} ) . \]
	Then as in the proof of \cref{lem:characterization_nilpotent_almost_perfect_modules}, we have identifications $\Coh^+(\fX) \simeq \Coh^+(A)$ and $\Coh^+( \fX^\rig ) \simeq \Coh^+(B)$, and under these identifications the functor $\rigg$ becomes equivalent to
	\[ - \otimes_A B \colon \Coh^+(A) \longrightarrow \Coh^+(B) . \]
	As $B \simeq A \otimes_{\kc} k$, we see that $A \to B$ is a Zariski open immersion.
	The conclusion now follows from \cite[Theorem 2.12]{Hennion_Porta_Vezzosi_Formal_gluing}.
\end{proof}

To complete the proof of the non-emptiness of $\FormalModels(\cF)$, it would be enough to know that the essential image of the functor $\Coh^+(\fX) \to \Coh^+(\fX^\rig)$ satisfies descent.
This is analogous to \cite[Theorem 7.3]{Hennion_Porta_Vezzosi_Formal_gluing}.

\begin{defin}
	Let $\fX$ be a derived $\kc$-adic \DM stack locally topologically almost of finite presentation.
	We define the stable $\infty$-category $\Coh^+_{\mathrm{loc}}( \fX )$ of \emph{$\mathfrak m$-local almost perfect complexes} as the cofiber
	\[ \Coh^+_{\mathrm{loc}}( \fX ) \coloneqq \cofib \left( \Coh^+_{\mathrm{nil}}( \fX ) \hookrightarrow \Coh^+( \fX ) \right) . \]
	We denote by $\rL \colon \Coh^+(\fX) \to \Coh^+_{\mathrm{loc}}(\fX)$ the canonical functor.
	We refer to $\rL$ as the \emph{localization functor}.
\end{defin}

We summarize below the formal properties of $\mathfrak m$-local almost perfect complexes:

\begin{prop} \label{prop:local_almost_perfect_complexes_formal_properties}
	Let $\fX$ be a derived $\kc$-adic \DM-stack locally topologically almost of finite presentation.
	Then:
	\begin{enumerate}
		\item there exists a unique $t$-structure on the stable $\infty$-category $\Coh^+_{\mathrm{loc}}( \fX )$ having the property of making the localization functor
		\[ \rL \colon \Coh^+(\fX) \longrightarrow \Coh^+_{\mathrm{loc}}(\fX) \]
		$t$-exact.
		
		\item The functor $\rigg \colon \Coh^+(\fX) \to \Coh^+(\fX^\rig)$ factors as
		\[ \Lambda \colon \Coh^+_{\mathrm{loc}}( \fX ) \longrightarrow \Coh^+( \fX^\rig ) . \]
		Moreover, the essential images of $\rigg$ and $\Lambda$ coincide.
		
		\item If $\fX$ is affine, then the functor $\Lambda$ is an equivalence.
	\end{enumerate}
\end{prop}

\begin{proof}
	We start by proving (1).
	Using \cite[Corollary 2.9]{Hennion_Porta_Vezzosi_Formal_gluing} we have to check that the $t$-structure on $\Coh^+(\fX)$ restricts to a $t$-structure on $\Coh^+_{\mathrm{nil}}(\fX)$ and that the inclusion
	\[ i \colon \Cohh_{\mathrm{nil}}(\fX) \longhookrightarrow \Cohh(\fX) \]
	admits a right adjoint $R$ whose counit $i( R( X ) ) \to X$ is a monomorphism for every $X \in \Cohh(\fX)$.
	For the first statement, we remark that it is enough to check that the functor $\rigg \colon \Coh^+( \fX ) \to \Coh^+( \fX^\rig )$ is $t$-exact.
	As both $\Coh^+( \fX )$ and $\Coh^+( \fX^\rig )$ satisfy \'etale descent in $\fX$, we can test this locally on $\fX$.
	When $\fX$ is affine, the assertion follows directly from \cref{prop:modules_on_affines_and_affinoid}.
	As for the second statement, we first observe that
	\[ \Cohh( \fX ) \simeq \Cohh( \trunc( \fX ) ) . \]
	We can therefore assume that $\fX$ is underived.
	At this point, the functor $R$ can be explicitly described as the functor sending $\fF\in \Cohh( \fX )$ to the subsheaf of $\fF$ spanned by $\mathfrak m$-nilpotent sections.
	The proof of (1) is thus complete.
	
	We now turn to the proof of (2).
	The existence of $\Lambda$ and the factorization $\rigg \simeq \Lambda \circ \rL$ follow from the definitions.
	Moreover, $\rL \colon \Coh^+( \fX ) \to \Coh^+_{\mathrm{loc}}( \fX )$ is essentially surjective (cf.\ \cite[Lemma 2.3]{Hennion_Porta_Vezzosi_Formal_gluing}).
	It follows that the essential images of $\rigg$ and of $\Lambda$ coincide.
	
	Finally, (3) follows directly from \cref{prop:modules_on_affines_and_affinoid} and \cite[Theorem 2.12]{Hennion_Porta_Vezzosi_Formal_gluing}.
\end{proof}

The commutativity of \eqref{eq:naturality_rigg} implies that a morphism $\ff \colon \fX \to \fY$ in $\dfDM_{\kc}^{\mathrm{taft}}$ induces a well defined functor
\[ \ff^{\circ *} \colon \Coh^+_{\mathrm{loc}}( \fY ) \longrightarrow \Coh^+_{\mathrm{loc}}( \fX ) . \]
It is a simple exercise in $\infty$-categories to promote this construction to an actual functor
\[ \Coh^+_{\mathrm{loc}} \colon \big( \dfDM_{\kc}^{\mathrm{taft}} \big)\op \longrightarrow \Cat_\infty^{\mathrm{st}} . \]
Having \cref{lem:existence_formal_models_affine} and \cref{prop:local_almost_perfect_complexes_formal_properties} at our disposal, the question of the non-emptiness of $\FormalModels(\cF)$ is essentially reduced to the the following:

\begin{thm} \label{thm:descent_punctured_category}
	Let $\dfSch_{\kc}^{\mathrm{taft}, \mathrm{qcqs}}$ denote the $\infty$-category of derived $\kc$-adic schemes which are quasi-compact, quasi separated and topologically almost of finite presentation.
	Then the functor
	\[ \Coh^+_{\mathrm{loc}} \colon \big( \dfSch_{\kc}^{\mathrm{taft}, \mathrm{qcqs}} \big)\op \longrightarrow \Cat_\infty^{\mathrm{st}}  \]
	is a hypercomplete sheaf for the formal Zariski topology.
\end{thm}

\begin{proof}
	A standard descent argument reduces us to prove the following statement: let $\ff_\bullet \colon \fU_\bullet \to \fX$ be a derived affine $\kc$-adic Zariski hypercovering.
	Then the canonical map
	\begin{equation} \label{eq:descent_functor}
		\ff_\bullet^{\circ *} \colon \Coh^+_{\mathrm{loc}}( \fX ) \longrightarrow \lim_{ [n] \in \mathbf \Delta } \Coh^+_{\mathrm{loc}}( \fU_\bullet )
	\end{equation}
	is an equivalence.
	Using \cite[Lemma 3.20]{Hennion_Porta_Vezzosi_Formal_gluing} we can endow the right hand side with a canonical $t$-structure.
	It follows from the characterization of the $t$-structure on $\Coh^+_{\mathrm{loc}}(\fX)$ given in \cref{prop:local_almost_perfect_complexes_formal_properties} that $\ff_\bullet^{\circ *}$ is $t$-exact.
	
	We will prove in \cref{cor:full_faithfulness} that $\ff_\bullet^{\circ *}$ is fully faithful.
	Assuming this fact, we can complete the proof as follows.
	We only need to check that $\ff_\bullet^{\circ *}$ is essentially surjective.
	Let $\cC$ be the essential image of $\ff_\bullet^{\circ *}$.
	We now make the following observations:
	\begin{enumerate}
		\item the heart of $\lim_{\mathbf \Delta} \Coh^+_{\mathrm{loc}}( \fU_\bullet )$ is contained in $\cC$.
		Indeed, \cref{lem:existence_formal_models_affine} implies that
		\[ \Lambda_n \colon \Coh^+_{\mathrm{loc}}( \fU_n ) \longrightarrow \Coh^+( \fU_n^\rig ) \]
		is an equivalence.
		These equivalences induce a $t$-exact equivalence
		\begin{equation} \label{eq:computing_the_limit}
			\Coh^+( \fX^\rig ) \simeq \lim_{ [n] \in \mathbf \Delta } \Coh^+_{\mathrm{loc}}( \fU_\bullet ) .
		\end{equation}
		Passing to the heart and using the canonical equivalences
		\[ \Cohh_{\mathrm{loc}}( \fX ) \simeq \Cohh_{\mathrm{loc}}( \trunc(\fX) ) \quad , \quad \Cohh( \fX^\rig ) \simeq \Cohh( \trunc( \fX^\rig ) ) , \]
		we can invoke the classical Rayanaud's theorem on formal models of coherent sheaves to deduce that the heart of the target of $\ff_\bullet^{\circ *}$ is contained in its essential image.
		
		\item The subcategory $\cC$ is stable.
		Indeed, let
		\[ \begin{tikzcd}[column sep = small]
			\cF' \arrow{r}{\varphi} & \cF \arrow{r}{\psi} & \cF''
		\end{tikzcd} \]
		be a fiber sequence in $\Coh^+( \fX^\rig ) \simeq \lim_{\mathbf \Delta} \Coh^+_{\mathrm{loc}}( \fU_\bullet )$ and suppose that two among $\cF$, $\cF'$ and $\cF''$ belong to $\cC$.
		Without loss of generality, we can assume that $\cF$ and $\cF''$ belong to $\cC$.
		Then choose elements $\fF$ and $\fF''$ in $\Coh^+_{\mathrm{loc}}( \fX )$ representing $\cF$ and $\cF''$.
		Since $\ff_\bullet^{\circ *}$ is fully faithful, we can find a morphism $\widetilde{\psi} \colon \fF \to \fF''$ lifting $\psi$.
		Set
		\[ \fF' \coloneqq \fib( \widetilde{\psi} \colon \fF \to \fF'' ) . \]
		Then $\Lambda( \fF' ) \simeq \cF'$, which means that under the equivalence \eqref{eq:computing_the_limit} the object $\cF'$ belongs to $\cC$.
	\end{enumerate}
	These two points together imply that $\ff_\bullet^{\circ *}$ is essentially surjective on cohomologically bounded elements.
	As both the $t$-structures on source and target of $\ff_\bullet^*$ are left $t$-complete, the conclusion follows.
\end{proof}

\begin{cor} \label{cor:comparing_local_and_rigid_almost_perfect_complexes}
	Let $\fX \in \dfSch_{\kc}^{\mathrm{taft}}$ and assume moreover that $\fX$ is quasi-compact and quasi-separated.
	Then the canonical map
	\[ \Lambda \colon \Coh^+_{\mathrm{loc}}( \fX ) \longrightarrow \Coh^+( \fX^\rig ) \]
	introduced in \cref{prop:local_almost_perfect_complexes_formal_properties} is an equivalence.
\end{cor}

\begin{proof}
	Let $\ff_\bullet \colon \fU_\bullet \to \fX$ be a derived affine $\kc$-adic Zariski hypercover.
	Consider the induced commutative diagram
	\[ \begin{tikzcd}
	\Coh^+_{\mathrm{loc}}( \fX ) \arrow{r}{\ff_\bullet^*} \arrow{d}{\Lambda} & \lim_{[n] \in \mathbf \Delta} \Coh^+_{\mathrm{loc}}( \fU_n ) \arrow{d}{ \Lambda_\bullet} \\
	\Coh^+( \fX^\rig ) \arrow{r}{f_\bullet^*} & \lim_{[n] \in \mathbf \Delta} \Coh^+( \fU_n^\rig ) ,
	\end{tikzcd} \]
	where we set $f_\bullet \coloneqq (\ff_\bullet)^\rig$.
	The right vertical map is an equivalence thanks to \cref{prop:local_almost_perfect_complexes_formal_properties}.
	On the other hand, $\Coh^+( \fX^\rig )$ satisfies descent in $\fX$, and therefore the bottom horizontal map is also an equivalence.
	Finally, \cref{thm:descent_punctured_category} implies that the top horizontal map is an equivalence as well.
	We thus conclude that $\Lambda \colon \Coh^+_{\mathrm{loc}}( \fX ) \to \Coh^+( \fX^\rig )$ is an equivalence.
\end{proof}

\begin{cor} \label{cor:existence_formal_models}
	Let $\fX \in \dfSch_{\kc}^{\mathrm{taft}}$ and assume moreover that it is quasi-compact and quasi-separated.
	For any $\cF \in \Coh^+( \fX^\rig )$, the $\infty$-category $\FormalModels( \cF )$ is non-empty.
\end{cor}

\begin{proof}
	The localization functor $\rL \colon \Coh^+( \fX ) \to \Coh^+_{\mathrm{loc}}( \fX )$ is essentially surjective by construction.
	Since $\fX$ is a quasi-compact and quasi-separted derived $\kc$-adic scheme topologically of finite presentation, \cref{cor:comparing_local_and_rigid_almost_perfect_complexes} implies that $\Lambda \colon \Coh^+_{\mathrm{loc}}( \fX ) \to \Coh^+( \fX^\rig )$ is an equivalence.
	The conclusion follows.
\end{proof}

\subsection{Proof of \cref{thm:descent_punctured_category}: fully faithfulness}

The only missing step in the proof of \cref{thm:descent_punctured_category} is the full faithfulness of the functor \eqref{eq:descent_functor}.
We will address this question by passing to the $\infty$-categories of ind-objects.
Let $\fX$ be a quasi-compact and quasi-separated derived $\kc$-adic scheme locally topologically almost of finite presentation.
\[ \ff \colon \fU \longrightarrow \fX \]
be a formally \'etale morphism.
Then $\ff$ induces a commutative diagram
\[ \begin{tikzcd}
	 \Ind( \Coh^+( \fX ) ) \arrow{d}{\ff^*} \arrow{r}{\rL_{\fX}} & \Ind( \Coh^+_{\mathrm{loc}}( \fX ) ) \arrow{d}{\ff^{\circ *}} \\
	 \Ind( \Coh^+( \fU ) ) \arrow{r}{\rL_{\fU}} & \Ind( \Coh^+_{\mathrm{loc}}( \fU ) ) .
\end{tikzcd} \]
The functors $\ff^*$ and $\ff^{\circ *}$ commute with colimits, and therefore they admit right adjoints $\ff_*$ and $\ff^\circ_*$.
In particular, we obtain a Beck-Chevalley transformation
\begin{equation} \label{eq:Beck_Chevalley_I}
	\theta \colon \rL_\fX \circ \ff_* \longrightarrow \ff^\circ_* \circ \rL_\fU .
\end{equation}
A key step in the proof of the full faithfulness of the functor \eqref{eq:descent_functor} is to verify that $\theta$ is an equivalence when evaluated on objects in $\Cohh( \fU )$.
Let us start with the following variation on \cite[Lemma 7.14]{Hennion_Porta_Vezzosi_Formal_gluing}:

\begin{lem} \label{lem:Beck_Chevalley_Verdier_quotient}
	Let
	\begin{equation} \label{eq:Beck_Chevalley_quotient}
		\begin{tikzcd}
			\cK_\cC \arrow[hook]{r}{i_\cC} \arrow{d}{F_\cK} & \cC \arrow{r}{L_\cC} \arrow{d}{F} & \cQ_\cC \arrow{d}{F_\cQ} \\
			\cK_\cD \arrow[hook]{r}{i_\cD} & \cD \arrow{r}{L_\cD} & \cQ_\cD 
		\end{tikzcd}
	\end{equation}
	be a diagram of stable $\infty$-categories and exact functors between them.
	Assume that:
	\begin{enumerate}
		\item the functors $i_\cC$ and $i_\cD$ are fully faithful and admit right adjoints $R_\cC$ and $R_\cD$, respectively;
		\item the functors $L_\cC$ and $L_\cD$ admit fully faithful right adjoints $j_\cC$ and $j_\cD$, respectively;
		\item the rows are fiber and cofiber sequences in $\Cat_\infty^{\mathrm{st}}$;
		\item the functors $F$, $F_\cK$ and $F_\cQ$ admit right adjoints $G$, $G_\cK$ and $G_\cQ$, respectively.
	\end{enumerate}
	Let $X \in \cD$ be an object.
	Then the following statements are equivalent:
	\begin{enumerate}
		\item the Beck-Chevalley transformation
		\[ q_X \colon L_\cC( G(X) ) \longrightarrow G_\cQ( L_\cD( X ) ) \]
		is an equivalence;
		\item the Beck-Chevalley transformation
		\[ \kappa_{R_\cD(X)} \colon i_\cC( G_\cK( R_\cD(X) ) ) \longrightarrow G( i_\cD( R_\cD( X ) ) ) \]
		is an equivalence.
	\end{enumerate}
\end{lem}

\begin{proof}
	Since $j_\cC$ and $i_\cC$ are fully faithful, it is equivalent to check that
	\[ j_\cC ( L_\cC( G(X) ) ) \longrightarrow j_\cC ( G_\cQ( L_\cD( X ) ) ) \]
	is an equivalence if and only if $\kappa_{R_\cD(X)}$ is an equivalence.
	Using the natural equivalences
	\[ j_\cC \circ G \simeq G j_\cD \quad , \quad G_\cK \circ R_\cD \simeq R_\cC \circ G  \]
	we obtain the following commutative diagram
	\[ \begin{tikzcd}
		i_\cC ( R_\cC ( G( X ) ) ) \arrow{r} \arrow{d} & G( X ) \arrow{r} \arrow[equal]{d} & j_\cC( L_\cC( G( X ) ) ) \arrow{d} \\
		G( i_\cD( R_\cD( X ) ) ) \arrow{r} & G( X ) \arrow{r} & G( j_\cD( L_\cD( X ) ) ) .
	\end{tikzcd} \]
	Moreover, since the rows of the diagram \eqref{eq:Beck_Chevalley_quotient} are Verdier quotients, we conclude that the rows in the above diagram are fiber sequences.
	Therefore, the leftmost vertical arrow is an equivalence if and only if the rightmost one is.
\end{proof}

\begin{lem} \label{lem:Beck_Chevalley_I}
	The Beck-Chevalley transformation \eqref{eq:Beck_Chevalley_I} is an equivalence whenever evaluated on objects in $\Cohh( \fU )$.
\end{lem}

\begin{proof}
	Using \cref{lem:Beck_Chevalley_Verdier_quotient}, we see that it is enough to prove that the Beck-Chevalley transformation associated to the square
	\[ \begin{tikzcd}
		\Ind( \Coh^+_{\mathrm{nil}}( \fX ) ) \arrow{d}{\ff^*} \arrow{r} & \Ind( \Coh^+( \fX ) ) \arrow{d}{\ff^*} \\
		\Ind( \Coh^+_{\mathrm{nil}}( \fU ) ) \arrow{r} & \Ind( \Coh^+( \fU ) )
	\end{tikzcd} \]
	is an equivalence when evaluated on objects of $\Cohh_{\mathrm{nil}}( \fU )$.
	As the horizontal functors are fully faithful, it is enough to check that the functor
	\[ \ff_* \colon \Ind( \Coh^+( U ) ) \longrightarrow \Ind( \Coh^+( \fX ) ) \]
	takes $\Cohh_{\mathrm{nil}}( \fU )$ to $\Ind( \Coh^+_{\mathrm{nil}}( \fX ) )$.
	Let $\fF \in \Cohh_{\mathrm{nil}}( \fU )$.
	We have to verify that $( \ff_*( \fF ) )^\rig \simeq 0$.
	Since $\fF$ is coherent and in the heart and since $\fU$ is quasi-compact we see that there exists an element $a \in \mathfrak m$ such that the map $\mu_a \colon \fF \to \fF$ given by multiplication by $a$ is zero.
	Therefore $\ff_*( \mu_a ) \colon \ff_*( \fF ) \to \ff_*( \fF )$ is homotopic to zero.
	Since $\ff_*( \mu_a )$ is equivalent to the endomorphism $\ff_*( \fF )$ given by multiplication by $a$, we conclude that $( \ff_*( \fF ) )^\rig \simeq 0$.
	The conclusion follows.
\end{proof}

Having these adjointability statements at our disposal, we turn to the actual study of the full faithfulness of the functor \eqref{eq:descent_functor}.
Let
\[ \fU_\bullet \colon \mathbf \Delta\op \longrightarrow \dfSch_{\kc}^{\mathrm{taft}} \]
be an affine $\kc$-adic Zariski hypercovering of $\fX$ and let $\ff_\bullet \colon \fU_\bullet \to \fX$ be the augmentation morphism.
The morphism $\ff_\bullet$ induces functors
\[ \ff_\bullet^* \colon \Ind( \Coh^+( \fX ) ) \longrightarrow \lim_{ [n] \in \mathbf \Delta } \Ind( \Coh^+( \fU_n ) ) \]
and
\[ \ff_{\bullet}^{\circ *} \colon \Ind( \Coh^+_{\mathrm{loc}}( \fX ) ) \longrightarrow \lim_{[n] \in \mathbf \Delta} \Ind( \Coh^+_{\mathrm{loc}}( \fU_n ) ) . \]
These functors commute by construction with filtered colimits, and therefore they admit right adjoints, that we denote respectively as
\[ \ff_{\bullet*} \colon \lim_{ [n] \in \mathbf \Delta } \Ind( \Coh^+( \fU_n ) ) \longrightarrow \Ind( \Coh^+( \fX ) )  \]
and
\[ \ff^{\circ}_{\bullet *} \colon \lim_{ [n] \in \mathbf \Delta } \Ind( \Coh^+_{\mathrm{loc}}( \fU_n ) ) \longrightarrow \Ind( \Coh^+_{\mathrm{loc}}( \fX ) ) . \]
Moreover, the functors $\ff_\bullet^*$ and $\ff_\bullet^{\circ *}$ fit in the following commutative diagram:
\[ \begin{tikzcd}
	\Ind( \Coh^+( \fX ) ) \arrow{d}{\rL} \arrow{r}{\ff_\bullet^*} & \lim_{ [n] \in \mathbf \Delta } \Ind( \Coh^+( \fU_\bullet ) ) \arrow{d}{\rL_\bullet} \\
	\Ind( \Coh^+_{\mathrm{loc}}( \fX ) ) \arrow{r}{\ff_\bullet^{\circ *}} & \lim_{ [n] \in \mathbf \Delta } \Ind( \Coh^+_{\mathrm{loc}}( \fU_\bullet ) ) .
\end{tikzcd} \]
In particular, we have an associated Beck-Chevalley transformation
\begin{equation} \label{eq:Beck_Chevalley}
	\theta \colon \rL \circ \ff_{\bullet *} \longrightarrow \ff_{\bullet *}^\circ \circ \rL_\bullet .
\end{equation}

\begin{prop} \label{prop:Beck_Chevalley}
	The Beck-Chevalley transformation \eqref{eq:Beck_Chevalley} is an equivalence when restricted to the full subcategory $\lim_{\mathbf \Delta} \Cohh( \fU_\bullet )$ of $\lim_{\mathbf \Delta} \Ind( \Coh^+( \fU_\bullet ) )$.
\end{prop}

\begin{proof}
	The discussion right after \cite[Corollary 8.6]{Porta_Yu_Higher_analytic_stacks_2014} allows us to identify the functor
	\[ \ff_{\bullet *} \colon \lim_{ [n] \in \mathbf \Delta} \Ind( \Coh^+( \fU_n ) ) \longrightarrow \Ind( \Coh^+( \fX ) ) \]
	with the functor informally described by sending a descent datum $\fF_\bullet \in \lim_{ \mathbf \Delta } \Ind( \Coh^+( \fU_\bullet ) )$ to
	\[ \lim_{ [n] \in \mathbf \Delta } \ff_{n*} \fF_n \in \Ind( \Coh^+( \fX ) ) . \]
	Similarly, the functor $\ff^\circ_{\bullet *}$ sends a descent datum $\cF_\bullet \in \lim_{\mathbf \Delta} \Ind( \Coh^+_{\mathrm{loc}}( \fU_\bullet )$ to
	\[ \lim_{ [n] \in \mathbf \Delta } \ff_{n*}^\circ \cF_n \in \Ind( \Coh^+_{\mathrm{loc}}( \fX ) ) . \]
	We therefore have to show that the Beck-Chevalley transformation
	\[ \theta \colon \rL \bigg( \lim_{ [n] \in \mathbf \Delta } \ff_{n*} \fF_n \bigg) \longrightarrow \lim_{ [n] \in \mathbf \Delta }  \ff^\circ_{n*}  (\rL_n \fF_n) \]
	is an equivalence whenever each $\fF_n$ belongs to $\Cohh( \fU_n )$.
	First notice that the functors $\ff_{\bullet*}$ and $\ff^\circ_{\bullet *}$ are left $t$-exact.
	In particular, if $\fF_\bullet \in \lim_{\mathbf \Delta} \Ind( \Cohh( \fU_\bullet ) )$ then both $\rL \ff_{\bullet *}( \fF_\bullet )$ and $\ff^\circ_{\bullet *}( \fF_\bullet )$ are coconnective.
	As the $t$-structures on $\lim_{ \mathbf \Delta } \Ind( \Coh^+( \fU_\bullet ) )$ and on $\lim_{\mathbf \Delta} \Ind( \Coh^+_{\mathrm{loc}}( \fU_\bullet ) )$ are right $t$-complete, we conclude that it is enough to prove that $\pi_i( \theta )$ is an isomorphism for every $i \in \mathbb Z$.
	We now observe that for $m \ge i + 2$ we have
	\[ \pi_i \bigg( \lim_{ [n] \in \mathbf \Delta } \ff^\circ_{n*}( \rL_n \fF_n ) \bigg) \simeq \pi_i \bigg( \lim_{ [n] \in \mathbf \Delta_{\le m} } \ff^\circ_{n*}( \rL_n \fF_n ) \bigg) , \]
	and similarly
	\[ \pi_i \bigg( \rL \bigg( \lim_{ [n] \in \mathbf \Delta } \ff_{n*} \fF_n \bigg) \bigg) \simeq \rL \bigg( \pi_i \bigg( \lim_{ [n] \in \mathbf \Delta } \ff_{n*} \fF_n \bigg) \bigg) \simeq \rL \bigg( \pi_i \bigg( \lim_{ [n] \in \mathbf \Delta_{\le m} } \ff_{n*} \fF_{n*} \bigg) \bigg) . \]
	It is therefore enough to prove that for every $m \ge 0$ the canonical map
	\[ \rL\bigg( \lim_{ [n] \in \mathbf \Delta_{\le m} } \ff_{n *} \fF_n \bigg) \longrightarrow \lim_{ [n] \in \mathbf \Delta_{\le m} } \ff^\circ_{n*}( \rL_n  \fF_n ) \]
	is an equivalence.
	As $\rL$ commutes with finite limits, we are reduced to show that the canonical map
	\[ \rL ( \ff_{n*} \fF_n ) \longrightarrow \ff^\circ_{n*} ( \rL_n \fF_n ) \]
	is an equivalence whenever $\fF_n \in \Cohh( \fU_n )$, which follows from \cref{lem:Beck_Chevalley_I}.
\end{proof}

\begin{cor} \label{cor:full_faithfulness}
	Let $\fX$ and $\ff_\bullet \colon \fU_\bullet \to \fX$ be as in the above discussion.
	Then the functor
	\[ \ff_\bullet^{\circ *} \colon \Coh^+_{\mathrm{loc}}( \fX ) \longrightarrow \lim_{ [n] \in \mathbf \Delta } \Coh^+_{\mathrm{loc}}( \fU_n ) \]
	is fully faithful.
\end{cor}

\begin{proof}
	As the functor $\ff_\bullet^{\circ *}$ is $t$-exact and the $t$-structure on both categories is left complete, we see that it is enough to reduce ourselves to prove that $\ff_\bullet^*$ is fully faithful when restricted to $\Cohb_{\mathrm{loc}}(\fX)$.
	Consider the following commutative cube:
	\begin{equation}
		\begin{tikzcd}
			\Coh^+( \fX ) \arrow{rr}{\ff_{\bullet}^*} \arrow[hook]{dr} \arrow{dd}& & \lim_{[n] \in \mathbf \Delta} \Coh^+( \fU_n ) \arrow{dd} \arrow[hook]{dr} \\
			{} & \Ind \left( \Coh^+( \fX ) \right) \arrow[crossing over]{rr}[near start]{\ff_{\bullet}^{*}}  \arrow{dd}[near end]{\rL_{\fX}}& & \lim_{[n] \in \mathbf \Delta} \Ind \left( \Coh^+( \fU_n ) \right) \arrow{dd}{\rL_{\fU_\bullet}} \\
			\Coh^+_{\mathrm{loc}}( \fX )  \arrow{rr}[near start]{\ff_{\bullet}^{\circ *}} \arrow[hook]{dr} & & \lim_{[n] \in \mathbf \Delta} \Coh^+_{\mathrm{loc}}(\fU_n) \arrow[hook]{dr} \\
			{} & \Ind \left( \Coh^+_{\mathrm{loc}} (\fX) \right) \arrow[leftarrow,crossing over]{uu} \arrow{rr}{\ff_{\bullet}^{\circ *}} & & \lim_{[n] \in \mathbf \Delta} \Ind \left(\Coh^+_{\mathrm{loc}} ( \fU_n ) \right) .
		\end{tikzcd}
	\end{equation}
	First of all, we observe that the diagonal functors are all fully faithful.
	It is therefore enough to prove that the functor
	\[ \ff^{\circ *}_\bullet \colon \Ind( \Coh^+_{\mathrm{loc}}(( \fX ) ) \longrightarrow \lim_{ [n] \in \mathbf \Delta } \Ind( \Coh^+_{\mathrm{loc}}( \fU_n ) ) \]
	is fully faithful when restricted to $\Coh^+_{\mathrm{loc}}( \fX )$.
	As this functor admits a right adjoint $\ff^\circ_{\bullet *}$, it is in turn enough to verify that for every $\cF \in \Cohb_{\mathrm{loc}}( \cF )$ the unit transformation
	\[ \eta \colon \cF \longrightarrow \ff^\circ_{\bullet *} \ff^{\circ *}_\bullet( \cF ) \]
	is an equivalence.
	Proceeding by induction on the number of nonvanishing homotopy groups of $\cF$, we see that it is enough to deal with the case of $\cF \in \Cohh_{\mathrm{loc}}(\cF)$.
	
	As the functor $\rL_\fX \colon \Coh^+( \fX ) \to \Coh^+_{\mathrm{loc}}( \fX )$ is essentially surjective and $t$-exact, we can choose $\fF \in \Cohh(\fX)$ and an equivalence
	\[ \rL_\fX( \fF ) \simeq \cF . \]
	Moreover, the unit transformation
	\[ \fF \longrightarrow \ff_{\bullet*} \ff_\bullet^* \fF \]
	is an equivalence.
	It is therefore enough to check that the Beck-Chevalley transformation associated to the front square is an equivalence when evaluated on objects in $\lim_{\mathbf \Delta} \Cohh(\fU_n)$.
	This is exactly the content of \cref{prop:Beck_Chevalley}.
\end{proof}

\subsection{Categories of formal models}

Let $\fX \in \dfSch_{\kc}^{\mathrm{taft}}$ be a quasi-compact and quasi-separated derived $\kc$-adic scheme topologically almost of finite presentation.
We established in \cref{cor:existence_formal_models} that for any $\cF \in \Coh^+( \fX^\rig )$ the $\infty$-category of formal models $\FormalModels( \cF )$ is non-empty.
Actually, we can use \cref{cor:comparing_local_and_rigid_almost_perfect_complexes} to be more precise about the structure of $\FormalModels( \cF )$.
We are in particular interested in showing that it is filtered.
We start by recording the following immediate consequence of \cref{cor:comparing_local_and_rigid_almost_perfect_complexes}:

\begin{lem} \label{cor:fully_faithful_adjoint_to_rigg}
	Let $\fX \in \dfSch_{\kc}^{\mathrm{taft}}$ be a quasi-compact and quasi-separated derived $\kc$-adic scheme topologically almost of finite presentation.
	Then the functor
	\[ \rigg \colon \Ind( \Coh^+( \fX ) ) \longrightarrow \Ind( \Coh^+( \fX^\rig ) ) \]
	admits a right adjoint
	\[ j \colon \Ind( \Coh^+( \fX^\rig ) ) \longrightarrow \Ind( \Coh^+( \fX ) ) , \]
	which is furthermore fully faithful.
\end{lem}

\begin{proof}
	\Cref{cor:comparing_local_and_rigid_almost_perfect_complexes} implies that the functor $\rigg$ induces the equivalence
	\[ \Lambda \colon \Coh^+_{\mathrm{loc}}( \fX ) \stackrel{\sim}{\longrightarrow} \Coh^+( \fX^\rig ) . \]
	In other words, we see that the diagram
	\[ \begin{tikzcd}
		\Coh^+_{\mathrm{nil}}( \fX ) \arrow{r} \arrow{d} & \Coh^+( \fX ) \arrow{d}{\rigg} \\
		0 \arrow{r} & \Coh^+( \fX^\rig )
	\end{tikzcd} \]
	is a pushout diagram in $\Cat_\infty^{\mathrm{st}}$.
	Passing to ind-completions, we deduce that $\Ind( \Coh^+( \fX^\rig ) )$ is a Verdier quotient of $\Ind( \Coh^+( \fX ) )$.
	Applying \cite[Lemma 2.5 and Remark 2.6]{Hennion_Porta_Vezzosi_Formal_gluing} we conclude that $\Ind( \Coh^+( \fX^\rig ) )$ is an accessible localization of $\Ind( \Coh^+( \fX ) )$.
	As these categories are presentable, we deduce that the localization functor $\rigg$ admits a fully faithful right adjoint, as desired.
\end{proof}

\begin{notation}
	Let $\fX \in \dfDM_{\kc}$.
	Given $\cF, \cG \in \Ind( \Coh^+(\fX) )$ we write $\Hom_{\fX}(\cF, \cG) \in \Mod_{\kc}$ for the $\kc$-enriched stable mapping space in $\Ind(\Coh^+(\fX))$.
\end{notation}

\begin{lem} \label{lem:hom_to_nilpotent_is_nilpotent}
	Let $\fX \in \dfSch_{\kc}^{\mathrm{taft}}$ be a quasi-compact and quasi-separated derived $\kc$-adic scheme topologically almost of finite presentation.
	Let $\cF \in \Coh^+(\fX)$ and $\cG \in \Coh^+_{\mathrm{nil}}(\fX)$.
	Then
	\[ \Hom_\fX(\cF, \cG) \otimes_{\kc} k \simeq 0 . \]
	In other words, $\Hom_{\fX}(\cF, \cG)$ is $\mathfrak m$-nilpotent in $\Mod_{\kc}$.
\end{lem}

\begin{proof}
	Since $\fX$ is quasi-compact, we can find a finite formal Zariski cover $\mathfrak U_i = \Spf(A_i)$ by formal affine schemes.
	Let $\mathfrak U_\bullet$ be the \v{C}ech nerve.
	Since this is a formal Zariski cover, there exists $m \gg 0$ such that
	\[ \Hom_{\fX}(\cF, \cG) \simeq \lim_{[n] \in \mathbf \Delta_{\le m}} \Hom_{\fU_n}( \cF|_{\fU_n}, \cG|_{\fU_n} ) . \]
	Since the functor $- \otimes_{\kc} k \colon \Mod_{\kc} \to \Mod_k$ is exact, it commutes with finite limits.
	Therefore, we see that it is enough to prove that the conclusion holds after replacing $\fX$ by $\fU_m$.
	Since $\fX$ is quasi-compact and quasi-separated, we see that each $\fU_m$ is quasi-compact and separated.
	In other words, we can assume from the very beginning that $\fX$ is quasi-compact and separated.
	In this case, each $\fU_m$ will be formal affine, and therefore we can further reduce to the case where $\fX$ is formal affine itself.
	
	Assume therefore $\fX = \Spf(A)$.
	In this case, $\Coh^+(\fX) \simeq \Coh^+(A)$ lives fully faithfully inside $\Mod_A$.
	Notice that $A \to A \otimes_{\kc} k$ is a Zariski open immersion.
	Therefore,
	\[ \Hom_A(\cF, \cG) \otimes_{\kc} k \simeq \Hom_A(\cF, \cG) \otimes_A (A \otimes_{\kc} k) \simeq \Hom_A( \cF \otimes_A \kc, \cG \otimes_A \kc ) \simeq 0 . \]
	Thus, the proof is complete.
\end{proof}

\begin{cor} \label{cor:base_change_hom}
	Let $\fX$ be as in the previous lemma.
	Given $\cF, \cG \in \Coh^+(\fX)$, the canonical map
	\[ \Hom_{\fX}( \cF, \cG ) \otimes_{\kc} k \longrightarrow \Hom_{\fX^\rig}( \cF^\rig, \cG^\rig ) \]
	is an equivalence.
\end{cor}

\todo{It's kind of miraculous that we can prove this corollary also for $\cG$ unbounded. See \cite[6.5.3.7]{Lurie_SAG}. Check very carefully the proof.}

\begin{proof}
	Denote by $R \colon \Ind(\Coh^+(\fX)) \to \Ind(\Coh^+_{\mathrm{nil}}(\fX))$ the right adjoint to the inclusion
	\[ i \colon \Ind(\Coh^+_{\mathrm{nil}}(\fX)) \hookrightarrow \Ind(\Coh^+(\fX)) . \]
	Then for any $\cG \in \Coh^+(\fX)$ we have a fiber sequence
	\[ i R(\cG) \longrightarrow \cG \longrightarrow j( \cG^\rig ) . \]
	In particular, we obtain a fiber sequence
	\[ \Hom_{\fX}( \cF, i R(\cG) ) \longrightarrow \Hom_{\fX}( \cF, \cG ) \longrightarrow \Hom_{\fX}( \cF, j( \cG^\rig ) ) . \]
	Now observe that
	\[ \Hom_{\fX}( \cF, j( \cG^\rig) ) \simeq \Hom_{\fX^\rig}(\cF^\rig, \cG^\rig ) . \]
	Notice also that since $\kc \to k$ is an open Zariski immersion, $\Hom_{\fX^\rig}(\cF^\rig, \cG^\rig) \otimes_{\kc} k \simeq \Hom_{\fX^\rig}(\cF^\rig, \cG^\rig)$.
	In particular, applying $- \otimes_{\kc} k \colon \Mod_{\kc} \to \Mod_k$ we find a fiber sequence
	\[ \Hom_{\fX}( \cF, i R(\cG ) ) \otimes_{\kc} k \longrightarrow \Hom_{\fX}(\cF, \cG) \otimes_{\kc} k \longrightarrow \Hom_{\fX^\rig}(\cF^\rig, \cG^\rig) . \]
	It is therefore enough to check that $\Hom_{\fX}(\cF, i R(\cG)) \otimes_{\kc} k \simeq 0$.
	Since $i$ is a left adjoint, we can write
	\[ i R(\cG) \simeq \colim_{\alpha \in I} \cG_\alpha , \]
	where $I$ is filtered and $\cG_\alpha \in \Coh^+_{\mathrm{nil}}(\fX)$.
	As $\cF$ is compact in $\Ind(\Coh^+(\fX))$, we find
	\[ \Hom_{\fX}( \cF, i R(\cG) ) \otimes_{\kc} k \simeq \left( \colim_{\alpha \in I} \Hom_{\fX}( \cF, \cG_\alpha ) \right) \otimes_{\kc} k \simeq \colim_{\alpha \in I} \Hom_{\fX}(\cF, \cG_\alpha) \otimes_{\kc} k . \]
	Since each $\cG_\alpha$ belongs to $\Coh^+_{\mathrm{nil}}(\fX)$, \cref{lem:hom_to_nilpotent_is_nilpotent} implies that $\Hom_{\fX}(\cF, \cG_\alpha) \otimes_{\kc} k \simeq 0$.
	The conclusion follows.
\end{proof}

\begin{rem}
	Notice that \cref{cor:base_change_hom} holds without no bounded conditions on the cohomological amplitude on the considered almost perfect complexes. The key ingredient is the fact that the morphism $\Spec k \hookrightarrow \Spec \kc$ is an open immersion. Compare with \cite[Lemma 6.5.3.7]{Lurie_SAG}.
\end{rem}

\begin{construction}
	Choose generators $t_1, \ldots, t_n$ for $\mathfrak m$.
	We consider $\mathbb N^n$ as a poset with order given by
	\[ (m_1, \ldots, m_n) \le (m_1', \ldots, m_n') \iff m_1 \le m_1', m_2 \le m_2' , \ldots , m_n \le m_n' \]
	Introduce the functor
	\[ K \colon \mathbb N^n\longrightarrow \Ind( \Cohh( \Spf(\kc) ) ) \]
	defined as follows: $K$ sends every object to $\kc$, and it sends the morphism $\mathbf m \le \mathbf m'$ to multiplication by $t^{\mathbf m' - \mathbf m}$.
	By abuse of notation, we still denote the composition of $K$ with the inclusion $\Ind( \Cohh(\kc ) ) \to \Ind( \Coh^+(\kc) )$ by $K$.
	
	Let now $\fX \in \dfSch_{\kc}^{\mathrm{taft}}$ be a quasi-compact and quasi-separated derived $\kc$-adic scheme topologically almost of finite presentation.
	Let $\cF \in \Coh^+(\fX)$.
	The natural morphism $q \colon \fX \to \Spf(\kc)$ induces a functor
	\[ q^* \colon \Ind(\Coh^+(\Spf(\kc))) \longrightarrow \Ind( \Coh^+( \fX ) ) . \]
	We define the functor $K_\cF$ as
	\[ K_\cF \coloneqq q^*(K(-)) \otimes \cF \colon \mathbb N^n \longrightarrow \Ind(\Coh^+(\fX)) . \]
	We let $\cF^\loc$ denote the colimit of the functor $K_\cF$.
	
	Let $\cG \in \Coh^+(\fX^\rig)$ and let $\alpha \colon \cF^\rig \to \cG$ be a given map.
	Notice that the natural map
	\[ \cF^\rig \longrightarrow \colim_{\mathbb N^n} (K_\cF(-))^\rig \]
	is an equivalence.
	Therefore $\alpha$ induces a cone
	\[ (K_\cF(-))^\rig \longrightarrow \cG , \]
	which is equivalent to the given of a cone
	\[ K_\cF(-) \longrightarrow j( \cG ) . \]
	Specializing this construction for $\alpha = \mathrm{id}_{\cF^\rig}$, we obtain a canonical map
	\[ \gamma_\cF \colon \cF^\loc \longrightarrow j( \cF^\rig ) . \]
\end{construction}

\begin{lem} \label{lem:loc_nilpotent}
	Let $\fX \in \dfSch_{\kc}^{\mathrm{taft}}$ be a quasi-compact and quasi-separated derived $\kc$-adic scheme topologically almost of finite presentation.
	Let $\cF \in \Coh^+_{\mathrm{nil}}(\fX)$.
	Then $\cF^\loc \simeq 0$.
\end{lem}

\begin{proof}
	For any $\cG \in \Coh^+(\fX)$, we write $\Hom_{\fX}(\cG, \cF) \in \Mod_{\kc}$ for the $\kc$-enriched mapping space.
	As $\cG$ is compact in $\Ind(\Coh^+(\fX))$, we have
	\[ \Hom_{\fX}( \cG, \cF^\loc ) \simeq \colim_{\mathbb N^n} \Hom_\fX( \cG, K_\cF(-) ) \simeq \Hom_\fX( \cG, \cF ) \otimes_{\kc} k . \]
	\Cref{cor:base_change_hom} implies that
	\[ \Hom_\fX( \cG, \cF ) \otimes_{\kc} k \simeq \Hom_{\fX^\rig}( \cG^\rig, \cF^\rig ) \simeq 0 . \]
	It follows that $\cF^\loc \simeq 0$.
\end{proof}

\begin{lem} \label{lem:colimit_torsion_free}
	Let $\fX \in \dfSch_{\kc}^{\mathrm{taft}}$ be a quasi-compact and quasi-separated derived $\kc$-adic scheme topologically almost of finite presentation.
	Let $\cF \in \Coh^+(\fX)$.
	Then for any $\cG \in \Coh^+_{\mathrm{nil}}(\fX)$, one has
	\[ \Map_{\Ind(\Coh^+(\fX))}( \cG, \cF^\loc ) \simeq 0 . \]
\end{lem}

\begin{proof}
	It is enough to prove that for every $i \ge 0$ we have
	\[ \pi_i \Map_{\Ind(\Coh^+(\fX))}( \cG, \cF^\loc ) \simeq 0 . \]
	Up to replacing $\cF$ by $\cF[i]$, we see that it is enough to deal with the case $i = 0$.
	Let therefore $\alpha \colon \cG \to \cF^\loc$ be a representative for an element in $\pi_0 \Map_{\Ind(\Coh^+(\fX))}( \cG, \cF^\loc )$.
	As $\cG$ is compact in $\Ind(\Coh^+(\fX))$, the map $\alpha$ factors as $\alpha' \colon \cG \to \cF$, and therefore it induces a map $\widetilde{\alpha} \colon \cG^\loc \to \cF^\loc$ making the diagram
	\[ \begin{tikzcd}
		\cG \arrow{r}{\alpha'} \arrow{d} & \cF \arrow{d} \\
		\cG^\loc \arrow{r}{\widetilde{\alpha}} & \cF^\loc
	\end{tikzcd} \]
	commutative, where both compositions are equivalent to $\alpha$.
	Now, \cref{lem:loc_nilpotent} implies that $\cG^\loc \simeq 0$, and therefore $\alpha$ is nullhomotopic, completing the proof.
\end{proof}

\begin{lem} \label{lem:computation_j}
	Let $\fX \in \dfSch_{\kc}^{\mathrm{taft}}$ be a quasi-compact and quasi-separated derived $\kc$-adic scheme topologically almost of finite presentation.
	Let $\cF \in \Coh^+(\fX)$.
	Then the canonical map
	\[ \gamma_\cF \colon \cF^\loc \longrightarrow j( \cF^\rig ) \]
	is an equivalence.
\end{lem}

\begin{proof}
	Let $\cG \in \Coh^+_{\mathrm{nil}}(\fX)$.
	Then
	\[ \Map_{\Ind(\Coh^+(\fX))}( \cG, j( \cF^\rig ) ) \simeq \Map_{\Ind(\Coh^+(\fX^\rig))}( \cG^\rig, \cF ) \simeq 0 . \]
	\Cref{lem:colimit_torsion_free} implies that the same holds true replacing $j(\cF^\rig)$ with $\cF^\loc$.
	As $\Coh^+_{\mathrm{nil}}(\fX)$ is a stable full subcategory of $\Coh^+(\fX)$, it follows that
	\[ \Hom_{\fX}(\cG, j(\cF)) \simeq \Hom_{\fX}(\cG, \cF^\loc ) \simeq 0 . \]
	Let $\cH \coloneqq \fib( \gamma_\cF )$.
	Then for any $\cG \in \Coh^+_{\mathrm{nil}}(\fX)$, one has
	\[ \Hom_{\fX}(\cG, \cH) \simeq 0 . \]
	On the other hand,
	\[ \cH^\rig \simeq \fib( \gamma_\cF^\rig ) \simeq 0 . \]
	It follows that $\cH \in \Ind( \Coh^+_{\mathrm{nil}}(\fX) )$, and hence that $\cH \simeq 0$.
	Thus, $\gamma_\cF$ is an equivalence.
\end{proof}

\begin{thm} \label{thm:formal_models_filtered}
	Let $\fX \in \dfSch_{\kc}$ be a quasi-compact and quasi-separated derived $\kc$-adic scheme.
	Let $\cF \in \Coh^+( \fX^\rig )$.
	Then the $\infty$-category $\FormalModels(\cF)$ of formal models for $\cF$ is non-empty and filtered.
\end{thm}

\begin{proof}
	We know that $\FormalModels(\cF)$ is non-empty thanks to \cref{cor:existence_formal_models}.
	Pick one formal model $\cF \in \FormalModels(\cF)$.
	Then \cref{lem:computation_j} implies that the canonical map
	\[ \gamma_\cF \colon \cF^\loc \longrightarrow j( \cF^\rig ) \simeq j( \cF ) \]
	is an equivalence.
	We now observe that $\FormalModels(\cF)$ is by definition a full subcategory of
	\[ \Coh^+(\fX)_{/\cF} \coloneqq \Coh^+(\fX) \times_{\Ind(\Coh^+(\fX))} \Ind(\Coh^+(\fX))_{/j(\cF)} . \]
	As this $\infty$-category is filtered, it is enough to prove that every object $\cG \in \Coh^+(\fX)_{/\cF}$ admits a map to an object in $\FormalModels(\cF)$.
	Let $\alpha \colon \cG \to j(\cF)$ be the structural map.
	Using the equivalence $\gamma_\cF$ and the fact that $\cG$ is compact in $\Ind(\Coh^+(\fX))$, we see that $\alpha$ factors as $\cG \to \cF$, which belongs to $\FormalModels(\cF)$ by construction.
\end{proof}	

\begin{cor} \label{lift_p_maps}
	Let $X \in \dAnk$ and $f \colon \cF \to \cG$ be a morphism $\Coh^+(X)$. Suppose we are given a formal model $\fX$ for $X$ together with formal models $\cF, \cG \in \Coh^+(\fX)$ for $\cF$ and $\cG$, respectively.
	Then there exists a morphism $\ff \colon \cF' \to \cG' $ in the \infcat $\Coh^+(\fX)$ lifting
	\[ t_1^{m_1}  \dots   t_n^{m_n}   f \colon \cF \to \cG, \quad \text{in $\Coh^+(X)$} \]
	for suitable non-negative integers $m_1, \dots, m_n \geq 0$.
\end{cor}

\begin{proof}
	Any map $\cF \to \cG$ induces a map $\fF \to j(\cF) \to j(\cG)$.
	Using the equivalence $j(\cG) \simeq \cG^\loc$ and the fact that $\cF$ is compact in $\Ind(\Coh^+(\fX))$, we see that the map $\fF \to j(\cG)$ factors as $\fF \to \cG$.
	Unraveling the definition of the functor $K_\cG(-)$, we see that the conclusion follows.
\end{proof}

For later use, let us record the following consequence of \cref{lem:computation_j}:

\begin{cor} \label{cor:characterization_nilpotent}
	Let $\fX \in \dfSch_{\kc}^{\mathrm{taft}}$ be a quasi-compact and quasi-separated derived $\kc$-adic scheme topologically almost of finite presentation.
	Let $\cF \in \Coh^+(\fX)$.
	Then $\cF$ is $\mathfrak m$-nilpotent if and only if $\cF^\loc \simeq 0$.
\end{cor}

\begin{proof}
	If $\cF$ is $\mathfrak m$-nilpotent, the conclusion follows from \cref{lem:loc_nilpotent}.
	Suppose vice-versa that $\cF^\loc \simeq 0$.
	Then \cref{lem:computation_j} implies that
	\[ j( \cF^\rig ) \simeq \cF^\loc \simeq 0 . \]
	Now, \cref{cor:fully_faithful_adjoint_to_rigg} shows that $j$ is fully faithful.
	In particular it is conservative and therefore $\cF^\rig \simeq 0$.
	In other words, $\cF$ belongs to $\Coh^+_{\mathrm{nil}}(\fX)$.
\end{proof}

\section{Flat models for morphisms of derived analytic spaces}

Using the study of formal models for almost perfect complexes carried out in the previous section, we can prove the following derived version of \cite[Theorem 5.2]{Bosch_Formal_II}:

\begin{thm} \label{thm:formal_model_flat_map}
	Let $f \colon X \to Y$ be a proper map of quasi-paracompact derived \kanal spaces.
	Assume that:
	\begin{enumerate}
		\item the truncations of $X$ and $Y$ are \kanal spaces.\footnote{As opposed to \kanal \DM stacks.}
		\item The map $f$ is flat.
	\end{enumerate}
	Then there exists a proper flat formal model $\ff \colon \fX \to \fY$ in $\dfSch_{\kc}^{\mathrm{taft}}$ for $f$.
\end{thm}

\begin{proof}
	We construct, by induction on $n$, the following data:
	\begin{enumerate}
		\item derived $\kc$-adic schemes $\fX_n$ and $\fY_n$ equipped with equivalences
		\[ \fX_n^\rig \simeq \mathrm t_{\le n}(X) , \qquad \fY_n^\rig \simeq \mathrm t_{\le n}(Y) . \]
		\item Morphisms $\fX_n \to \fX_{n-1}$ and $\fY_n \to \fY_{n-1}$ exhibiting $\fX_{n-1}$ and $\fY_{n-1}$ as $(n-1)$-truncations of $\fX_n$ and $\fY_n$, respectively.
		\item A proper flat morphism $\ff_n \colon \fX_n \to \fY_n$ and homotopies making the cube
		\[ \begin{tikzcd}
			{} & \fX_n^\rig \arrow{rr}{\ff_n^\rig} \arrow{dd} \arrow{dl} & & \fY_n^\rig \arrow{dd} \arrow{dl} \\
			\fX_{n-1}^\rig \arrow[crossing over]{rr}[near end]{\ff_{n-1}^\rig} \arrow{dd} & & \fY_{n-1}^\rig \\
			{} & \mathrm t_{\le n}(X) \arrow{rr} \arrow{dl} & & \mathrm t_{\le n}(Y) \arrow{dl} \\
			\mathrm t_{\le n - 1}(X) \arrow{rr} & & \mathrm t_{\le n- 1}(Y) \arrow[leftarrow, crossing over]{uu}
		\end{tikzcd} \]
		commutative.
	\end{enumerate}
	Having these data at our disposal, we set
	\[ \fX \coloneqq \colim_n \fX_n , \qquad \fY \coloneqq \colim_n \fY_n , \]
	and we let $\ff \colon \fX \to \fY$ be map induced by the morphisms $\ff_n$.
	The properties listed above imply that $\ff$ is proper and flat and that its generic fiber is equivalent to $f$.
	
	We are therefore left to construct the data listed above.
	When $n = 0$, we can apply the flattening technique of Raynaud-Gruson (see \cite[Theorem 5.2]{Bosch_Formal_II}) to produce a proper flat formal model $\ff_0 \colon \fX_0 \to \fY_0$ for $\trunc(f) \colon \trunc(X) \to \trunc(Y)$.
	Assume now that we constructed the above data up to $n$ and let us construct it for $n+1$.
	Set $\cF \coloneqq \pi_{n+1}(\cO_X)[n+2]$ and $\cG \coloneqq \pi_{n+1}(\cO_Y)[n+2]$.
	Using \cite[Corollary 5.44]{Porta_Yu_Representability}, we can find analytic derivations $d_\alpha \colon (\mathrm t_{\le n}X)[\cF] \to \mathrm t_{\le n} X$ and $d_\beta \colon (\mathrm t_{\le n} Y)[\cG] \to \mathrm t_{\le n} Y$ making the following cube
	\begin{equation} \label{eq:flat_formal_model}
		\begin{tikzcd}
			{} & (\mathrm t_{\le n} X)[\cF] \arrow{rr}{d_0} \arrow{dd} \arrow{dl}{d_\alpha} & & \mathrm t_{\le n} X \arrow{dl} \arrow{dd}{f_n} \\
			\mathrm t_{\le n} X \arrow[crossing over]{rr} \arrow{dd} & & \mathrm t_{\le n + 1} X \\
			{} & (\mathrm t_{\le n} Y)[\cG] \arrow{rr}[near start]{d_0} \arrow{dl}{d_\beta} & & \mathrm t_{\le n} Y \arrow{dl} \\
			\mathrm t_{\le n} Y \arrow{rr} & & \mathrm t_{\le n + 1} Y \arrow[leftarrow, crossing over]{uu}[swap,near end]{f_{n+1}}
		\end{tikzcd}
	\end{equation}
	commutative.
	Here $d_0$ denote the zero derivation and we set $f_n \coloneqq \mathrm t_{\le n}(f)$, $f_{n+1} \coloneqq \mathrm t_{\le n+1}(f)$.
	The derivations $d_\alpha$ and $d_\beta$ correspond to morphisms $\alpha \colon \anL_{\mathrm t_{\le n} X} \to \cF$ and $\beta \colon \anL_{\mathrm t_{\le n} Y} \to \cG$, respectively.
	Moreover, the commutativity of the left side square in \eqref{eq:flat_formal_model} is equivalent to the commutativity of
	\[ \begin{tikzcd}
		f_n^* \anL_{\mathrm t_{\le n} Y} \arrow{r}{f_n^* \beta} \arrow{d} & f_n^* \cG \arrow{d} \\
		\anL_{\mathrm t_{\le n} X} \arrow{r}{\alpha} & \cF
	\end{tikzcd} \]
	in $\Coh^+(\mathrm t_{\le n} X)$.
	Notice that, since $f$ is flat, the morphism $f_n^* \cF \to \cG$ is an equivalence.
	Using \cref{thm:adic_and_analytic_cotangent_complex} and the induction hypothesis, we know that $\adL_{\mathfrak Y_n}$ is a canonical formal model for $\anL_{\mathrm t_{\le n} X}$.
	Using \cref{thm:formal_models_filtered}, we can therefore find a formal model $\overline{\beta} \colon \adL_{\mathfrak Y_n} \to \fG$ for the map $\beta$.
	We now set
	\[ \fF \coloneqq \ff_n^* \fG . \]
	Using \cref{lift_p_maps}, we can find $\mathbf m \in \mathbb N^n$ and a formal model $\talpha \colon \adL_{\mathfrak X_n} \to \fF$ for $t^{\mathbf m} \alpha$ together with a homotopy making the diagram
	\[ \begin{tikzcd}
		\ff_n^* \adL_{\mathfrak Y_n} \arrow{r}{t^{\mathbf m} \ff_n^* \overline{\beta}} \arrow{d} & \ff_n^* \fG \arrow[equal]{d} \\
		\adL_{\mathfrak X_n} \arrow{r}{\tilde{\alpha}} & \fF
	\end{tikzcd} \]
	commutative.
	Set $\tbeta \coloneqq t^{\mathbf m} \overline{\beta} \colon \adL_{\mathfrak Y_n} \to \fG$.
	Then $\talpha$ and $\tbeta$ induce a commutative square
	\begin{equation} \label{eq:flat_formal_model_II}
		\begin{tikzcd}
			\mathfrak X_n[\fF] \arrow{r}{d_{\tilde{\alpha}}} \arrow{d} & \mathfrak X_n \arrow{d}{\ff_n} \\
			\mathfrak Y_n[\fG] \arrow{r}{d_{\tilde{\beta}}} & \mathfrak Y_n .
		\end{tikzcd}
	\end{equation}
	We now define $\mathfrak X_{n+1}$ and $\mathfrak Y_{n+1}$ as the square-zero extensions associated to $\talpha$ and $\tbeta$.
	In other words, they are defined by the following pushout diagrams:
	\[ \begin{tikzcd}
		\mathfrak X_n[\fF] \arrow{d}{d_{\tilde{\alpha}}} \arrow{r}{d_0} & \mathfrak X_n \arrow{d} \\
		\mathfrak X_n \arrow{r} & \mathfrak X_{n+1}
	\end{tikzcd} , \quad \begin{tikzcd}
		\mathfrak Y_n[\fG] \arrow{r}{d_0} \arrow{d}{d_{\tilde{\beta}}} & \mathfrak Y_n \arrow{d} \\
		\mathfrak Y_n \arrow{r} & \mathfrak Y_{n+1} .
	\end{tikzcd} \]
	The commutativity of \eqref{eq:flat_formal_model_II} provides a canonical map $\ff_{n+1} \colon \fX_{n+1} \to \mathfrak Y_{n+1}$, which is readily verified to be proper and flat.
	We are therefore left to verify that $\ff_{n+1}$ is a formal model for $f_{n+1}$.
	Unraveling the definitions, we see that it is enough to produce equivalences $a \colon (\mathrm t_{\le n} X)[\cF] \xrightarrow{\sim} (\mathrm t_{\le n} X)[\cF]$ and $b \colon (\mathrm t_{\le n} Y)[\cG] \xrightarrow{\sim} (\mathrm t_{\le n} Y)[\cG]$ making the following diagrams
	\begin{equation} \label{eq:flat_formal_model_III}
		\begin{tikzcd}
			(\mathrm t_{\le n} X)[\cF] \arrow{r}{d_{t^{\mathbf m} \alpha}} \arrow{d}{a} & \mathrm t_{\le n} X \arrow[equal]{d} \\
			(\mathrm t_{\le n} X)[\cF] \arrow{r}{d_{\alpha}} & \mathrm t_{\le n} X
		\end{tikzcd} , \quad \begin{tikzcd}
			(\mathrm t_{\le n} Y)[\cG] \arrow{r}{d_{t^{\mathbf m} \beta}} \arrow{d}{b} & \mathrm t_{\le n} Y \arrow[equal]{d} \\
			(\mathrm t_{\le n} Y)[\cG] \arrow{r}{d_{\beta}} & \mathrm t_{\le n} Y
		\end{tikzcd}
	\end{equation}
	commutative.
	The situation is symmetric, so it is enough to deal with $\mathrm t_{\le n} X$.
	Consider the morphism
	\[ t^{- \mathbf m} \colon \cF \longrightarrow \cF , \]
	which exists because all the elements $t_i \in \mathfrak m$ are invertible in $k$.
	For the same reason it is an equivalence, with inverse given by multiplication by $t^{\mathbf m}$.
	This morphism induces a map
	\[ a \colon (\mathrm t_{\le n} X)[\cF] \longrightarrow ( \mathrm t_{\le n} X)[\cF] , \]
	which by functoriality is an equivalence.
	We now observe that the commutativity of \eqref{eq:flat_formal_model_III} is equivalent to the commutativity of
	\[ \begin{tikzcd}
		\anL_{\mathrm t_{\le n} X} \arrow{r}{t^{\mathbf m} \alpha} \arrow[equal]{d} & \cF \arrow{d}{t^{- \mathbf m}} \\
		\anL_{\mathrm t_{\le n} X} \arrow{r}{\alpha} & \cF ,
	\end{tikzcd} \]
	which is immediate.
	The proof is therefore achieved.
\end{proof}

\section{The plus pushforward for almost perfect sheaves}

Let $f \colon X \to Y$ be a proper map between derived \kanal spaces of finite tor-amplitude.
In \cite[Definition 7.9]{Porta_Yu_Mapping} it is introduced a functor
\[ f_+ \colon \Perf(X) \longrightarrow \Perf(Y) , \]
and it is shown in Proposition 7.11 in loc.\ cit.\ that for every $\cG \in \Coh^+(Y)$ there is a natural equivalence
\[ \Map_{\Coh^+(X)}( \cF, f^* \cG ) \simeq \Map_{\Coh^+(Y)}( f_+(\cF), \cG ) . \]
In this section we extend the definition of $f_+$ to the entire $\Coh^+(X)$, at least under the stronger assumption of $f$ being flat.

\begin{rem}
	In algebraic geometry, the extension of $f_+$ to $\Coh^+(X)$ passes through the extension to $\mathrm{QCoh}(X) \simeq \Ind(\Perf(X))$.
	This is ultimately requires being able to describe every element in $\Coh^+(X)$ as a filtered colimit of elements in $\Perf(X)$, which in analytic geometry is possible only locally.
\end{rem}

Therefore, this technique cannot be applied in analytic geometry.
When dealing with non-archimedean analytic geometry, formal models can be used to circumvent this problem.

\begin{prop} \label{prop:formal_plus_pushforward}
	Let $f \colon \fX \to \fY$ be a proper map between derived $\kc$-adic schemes.
	Assume that $f$ has finite tor amplitude.
	Then the functor
	\[ f^* \colon \Coh^+(\fY) \to \Coh^+(\fX) \]
	admits a left adjoint
	\[ f_+ \colon \Coh^+(\fX) \to \Coh^+(\fY) . \]
\end{prop}

\begin{proof}
	Let $\fX_n \coloneqq \fX \times_{\Spf(\kc)} \Spec(\kc / \mathfrak m^n)$ and define similarly $\fY_n$.
	Let $f_n \colon \fX_n \to \fY_n$ be the induced morphism.
	Then by definition of $\kc$-adic schemes, \personal{we are using the completeness in particular} we have
	\[ \fX \simeq \colim_{n \in \mathbb N} \fX_n , \qquad \fY \simeq \colim_{n \in \mathbb N} \fY_n , \]
	and therefore
	\[ \Coh^+(\fX) \simeq \lim_{n \in \mathbb N} \Coh^+(\fX_n) , \qquad \Coh^+(\fY) \simeq \lim_{n \in \mathbb N} \Coh^+(\fY_n) . \]
	Combining \cite[Remark 6.4.5.2(b) \& Proposition 6.4.5.4(1)]{Lurie_SAG}, we see that each functor
	\[ f_n^* \colon \Coh^+(\fY_n) \longrightarrow \Coh^+(\fX_n) \]
	admits a left adjoint $f_{n+}$.
	Moreover, Proposition 6.4.5.4(2) in loc.\ cit.\ implies that these functors $f_{n+}$ can be assembled into a natural transformation, and that therefore they induce a well defined functor
	\[ f_+ \colon \Coh^+(\fX) \longrightarrow \Coh^+(\fY) . \]
	Now let $\cF \in \Coh^+(\fX)$ and $\cG \in \Coh^+(\fY)$.
	Let $\cF_n$ and $\cG_n$ be the pullbacks of $\cF$ and $\cG$ to $\cX_n$ and $\cY_n$, respectively.
	Then
	\begin{align*}
		\Map_{\Coh^+(\fX)}( \cF, f^*(\cG) ) & \simeq \lim_{n \in \mathbb N} \Map_{\Coh^+(\fX_n)}( \cF_n, f_n^*( \cG_n) ) \\
		& \simeq \lim_{n \in \mathbb N} \Map_{\Coh^+(\fY_n)}( f_{n+}(\cF_n), \cG_n ) \\
		& \simeq \Map_{\Coh^+(\fY)}( f_+(\cF), \cG ) ,
	\end{align*}
	which completes the proof.
\end{proof}

\begin{cor} \label{cor:plus_pushforward}
	Let $f \colon X \to Y$ be a proper map between derived analytic spaces.
	Assume that $f$ is flat.
	Then the functor
	\[ f^* \colon \Coh^+(Y) \to \Coh^+(X) \]
	admits a left adjoint
	\[ f_+ \colon \Coh^+(X) \to \Coh^+(Y) . \]
\end{cor}

\begin{proof}
	Using \cref{thm:formal_model_flat_map}, we can choose a proper flat formal model $\ff \colon \fX \to \fY$ for $f$.
	Thanks to \cref{prop:formal_plus_pushforward}, we have a well defined functor
	\[ \ff_+ \colon \Coh^+(\fX) \longrightarrow \Coh^+(\fY) . \]
	We claim that it restricts to a functor
	\[ \ff_+ \colon \Coh^+_{\mathrm{nil}}(\fX) \longrightarrow \Coh^+_{\mathrm{nil}}(\fY) . \]
	Using \cref{cor:characterization_nilpotent}, it is enough to prove that
	\[ \ff_+( \cF )^\loc \simeq 0 . \]
	Extending $\ff_+$ to a functor $\ff_+ \colon \Ind( \Coh^+(\fX) ) \to \Ind( \Coh^+(\fY) )$, we see that
	\[ \ff_+( \cF )^\loc \simeq \ff_+(\cF^\loc) \simeq 0 . \]
	Using \cref{cor:comparing_local_and_rigid_almost_perfect_complexes}, we get a well defined functor
	\[ f_+ \colon \Coh^+(X) \longrightarrow \Coh^+(Y) . \]
	We only have to prove that it is left adjoint to $f^*$.
	Let $\cF \in \Coh^+(X)$ and $\cG \in \Coh^+(Y)$.
	Choose a formal model $\fF \in \Coh^+(\fX)$.
	Then unraveling the construction of $f_+$, we find a canonical equivalence \personal{Induced from the given equivalence $\fF^\rig \simeq \cF$}
	\[ f_+(\cF) \simeq \ff_+( \fF )^\rig . \]
	We now have the following sequence of natural equivalences:
	\begin{align*}
		\Map_{\Coh^+(Y)}( f_+( \cF ), \cG ) & \simeq \Map_{\Coh^+(Y)}( (\ff_+( \fF ))^\rig, \fG^\rig ) \\
		& \simeq \Map_{\Coh^+(\fX)}( \ff_+( \fF ), \fG ) \otimes_{\kc} k & \text{by \cref{cor:base_change_hom}} \\
		& \simeq \Map_{\Coh^+(\fX)}( \fF, \ff^* \fG ) \otimes_{\kc} k \\
		& \simeq \Map_{\Coh^+(\fX)}( \fF^\rig, ( \ff^* \fG )^\rig ) & \text{by \cref{cor:base_change_hom}} \\
		& \simeq \Map_{\Coh^+(\fX)}( \cF, f^* \cG ) .
	\end{align*}
	The proof is therefore complete.
\end{proof}

\begin{cor} \label{cor:plus_pushforward_base_change}
	Let $f \colon X \to Y$ be a proper and flat map between derived analytic spaces.
	Let $p \colon Z \to Y$ be any other map and consider the pullback square
	\[ \begin{tikzcd}
		W \arrow{r}{q} \arrow{d}{g} & X \arrow{d}{f} \\
		Z \arrow{r}{p} & Y .
	\end{tikzcd} \]
	Then for any $\cF \in \Coh^+(X)$ the canonical map
	\[ g_+(q^*(\cF)) \longrightarrow p^*(f_+(\cF)) \]
	is an equivalence.
\end{cor}

\begin{proof}
	Using \cref{thm:formal_model_flat_map}, we find a flat formal model $\ff \colon \fX \to \fY$.
	Choose a formal model $\mathfrak p \colon \fZ \to \fY$ for $p \colon Z \to Y$,
	\personal{It is always possible to find a formal model $\fZ \to \fY'$ for $p$.
		We can also find a zig-zag $\fY' \leftarrow \fY'' \to \fY$.
	Then $\fY'' \times_{\fY} \fX \to \fY''$ is a flat formal model for $f$ and $\fY'' \times_{\fY'} \fZ \to \fY''$ is a formal model for $p$.}
	and form the pullback square
	\[ \begin{tikzcd}
		\mathfrak W \arrow{r}{\mathfrak q} \arrow{d}{\mathfrak g} & \fX \arrow{d}{\ff} \\
		\mathfrak Z \arrow{r}{\mathfrak p} & \fY .
	\end{tikzcd} \]
	Choose also a formal model $\fF \in \Coh^+(\fX)$ for $\cF$.
	It is then enough to prove that the canonical map
	\[ \mathfrak g_+( \mathfrak q^*( \fF ) ) \longrightarrow \mathfrak p^*( \ff_+(\fF) ) \]
	is an equivalence.
	This follows at once by \cite[Proposition 6.4.5.4(2)]{Lurie_SAG}.
\end{proof}

\section{Representability of $\mathbf R \mathrm{Hilb}(X)$}

Let $p \colon X \to S$ be a proper and flat morphism of underived \kanal spaces.
We define the functor
\[ \rR \Hilb(X/S) \colon \dAfd_S\op \longrightarrow \cS \]
by sending $T \to S$ to the space of diagrams
\begin{equation} \label{eq:functor_points_Hilbert}
	\begin{tikzcd}[column sep = small]
		Y \arrow[hook]{rr}{i} \arrow{dr}[swap]{q_T} & & T \times_S X \arrow{dl}{p_T} \\
		{} & T
	\end{tikzcd}
\end{equation}
where $i$ is a closed immersion of derived \kanal spaces, and $q_T$ is flat.

\begin{prop} \label{prop:RHilb_cotangent_complex}
	Keeping the above notation and assumptions, $\rR \Hilb(X/S)$ admits a global analytic cotangent complex.
\end{prop}

\begin{proof}
	Let $x \colon T \to \rR\Hilb(X/S)$ be a morphism from a derived $k$-affinoid space $T \in \dAfd_S$.
	It classifies a diagram of the form \eqref{eq:functor_points_Hilbert}.
	Unraveling the definitions, we see that the functor
	\[ \DerAn_{\rR\Hilb(X/S),x}(T;-) \colon \Coh^+(T) \longrightarrow \rR\Hilb(X/S) \]
	can be explicitly written as
	\[ \DerAn_{\rR\Hilb(X/S),x}(T;\cF) \simeq \Map_{\Coh^+(Y)}( \anL_{Y/T \times_S X}, q_T^*( \cF ) ) . \]
	Since $q_T \colon Y \to T$ is proper and flat, \cref{cor:plus_pushforward} implies the existence of a left adjoint $q_{T+} \colon \Coh^+(Y) \to \Coh^+(T)$ for $q_T^*$.
	Moreover, \cite[Corollary 5.40]{Porta_Yu_Representability} implies that $\anL_{Y/ T \times_S X} \in \Coh{\ge 0}( Y )$.
	Therefore, we find
	\[ \DerAn_{\rR\Hilb(X/S),x}(T; \cF) \simeq \Map_{\Coh^+(T)}( q_{T+}( \anL_{Y / T \times_S X} ), \cF ) , \]
	and therefore $\rR \Hilb(X/S)$ admits an analytic cotangent complex at $x$.
	Using \cref{cor:plus_pushforward_base_change}, we see that it admits as well a global analytic cotangent complex.
\end{proof}

\begin{thm}
	Let $X$ be a \kanal space.
	Then $\mathbf{R} \mathrm{Hilb}(X)$ is a derived analytic space.
\end{thm}

\begin{proof}
	We only need to check the hypotheses of \cite[Theorem 7.1]{Porta_Yu_Representability}.
	The representability of the truncation is guaranteed by \cite[Proposition 5.3.3]{Conrad_Spreading-out}.
	The existence of the global analytic cotangent complex has been dealt with in \cref{prop:RHilb_cotangent_complex}.
	Convergence and infinitesimal cohesiveness are straightforward checks.
	The theorem follows.
\end{proof}

As a second concluding applications, let us mention that the theory of the plus pushforward developed in this paper allows to remove the lci assumption in \cite[Theorem 8.6]{Porta_Yu_Mapping}:

\begin{thm} \label{thm:rep_mapping}
	Let $S$ be a rigid \kanal space.
	Let $X,Y$ be rigid \kanal spaces over $S$.
	Assume that $X$ is proper and flat over $S$ and that $Y$ is separated over $S$.
	Then the $\infty$-functor $\bfMap_S(X,Y)$ is representable by a derived \kanal space separated over $S$.
\end{thm}

\begin{proof}
	The same proof of \cite[Theorem 8.6]{Porta_Yu_Mapping} applies.
	It is enough to observe that Corollaries \ref{cor:plus_pushforward} and \ref{cor:plus_pushforward_base_change} allow to prove Lemma 8.4 in loc.\ cit.\ by removing the assumption of $Y \to S$ being locally of finite presentation.
\end{proof}

\bibliographystyle{plain}
\bibliography{dahema}

\end{document}